\documentclass[]{amsart}
\usepackage{amsfonts, amsmath, amssymb, amscd, amsthm, mathrsfs}
\usepackage{amsmath, amsthm, amscd, amsfonts, amssymb, graphicx}
\usepackage[bookmarksnumbered, plainpages]{hyperref}

\newtheorem{theorem}{Theorem}[section]
\newtheorem{lemma}[theorem]{Lemma}
\newtheorem{proposition}[theorem]{Proposition}
\newtheorem{corollary}[theorem]{Corollary}
\newtheorem{definition}[theorem]{Definition}

\theoremstyle{example}

\theoremstyle{remark}
\newtheorem{remark}{Remark}[section]

\begin{document}

\title
[Gradient Flows of Higher Order Yang-Mills-Higgs Functionals]{Gradient Flows of Higher Order Yang-Mills-Higgs Functionals}

\author{Pan Zhang}
\address{School of Mathematics, Sun Yat-sen University, Guangzhou 510275, P.R. China}
\email{zhangpan5@mail.sysu.edu.cn}


\begin{abstract}
In this paper, we define a family of functionals generalizing the Yang-Mills-Higgs functional on a closed Riemannian manifold. Then we prove the short time existence of the corresponding gradient flow by a gauge fixing technique.  The lack of maximal principle for the higher order operator brings us a lot of inconvenience during the estimates for the Higgs field. We observe that the $L^2$-bound of the Higgs field is enough for energy estimates in $4$ dimension, and we show that, provided the order of derivatives, appearing in the higher order Yang-Mills-Higgs functionals, is strictly greater than 1, solutions to the gradient flow do not hit any finite time singularities. As for the Yang-Mills-Higgs $k$-functional with Higgs self-interaction, we show that, provided $\dim(M)<2(k+1)$, the associated gradient flow admits long time existence with smooth initial data. The proof depends on local $L^2$-derivative estimates, energy estimates and blow-up analysis.

\end{abstract}

\maketitle

\numberwithin{equation}{section}

\section{Introduction}

Let $(M,g)$ be a closed Riemannian manifold of real dimension $n$ and let $(E,h)$ be a vector bundle over $M$ with structure group $G$, where $G$ is a compact Lie group. The Yang-Mills functional, defined on the space of connections of $E$, is given by
\begin{equation*} \label{YMF}
\mathcal{YM}(\nabla)=\frac{1}{2}\int_M|F_{\nabla}|^2d\mathrm{vol}_g,
\end{equation*}
where $\nabla$ is a metric compatiable connection and $F_{\nabla}$ denotes its curvature.

$\nabla$ is called a Yang-Mills connection of $E$ if it satisfies the Yang-Mills equation
\begin{equation*} \label{YME}
D^*_{\nabla}F_{\nabla}=0.
\end{equation*}
A family of connections $\nabla_t$ is said to be a solution to the Yang-Mills flow if
\begin{equation*} \label{YMF}
\frac{\partial \nabla_t}{\partial t}=-D^*_{\nabla_t}F_{\nabla_t}.
\end{equation*}
The Yang-Mills flow was initially studied by Atiyah-Bott \cite{AB} and was suggested to understand the topology of the space of connections by infinite dimensional Morse theory.

In the case that $\nabla_t$ are compatible connections on a holomorphic bundle over a closed K\"{a}hler manifold. Owing to Donaldson \cite{Don85} and Simpson \cite{SIM}, the Yang-Mills flow exists smoothly for all time, and converges to a Hermitian-Yang-Mills connection on stable bundles. This results a correspondence, known as Hitchin-Kobayashi correspondence \cite{Hit,LT}, or Donaldson-Uhlenbeck-Yau Theorem \cite{Don87,UY86}. Natural generalizations to the unstable case have been obtained by Daskalopolos-Wentworth \cite{DW1,DW2}, Wilkin \cite{Wi}, Jacob \cite{Ja}, Sibley \cite{Si},  Li-Zhang-Zhang \cite{LZZ1,LZZ2}, Nie-Zhang \cite{NZ} and so on.

For general Riemannian context, the behavior of Yang-Mills flow is strongly influenced by the dimension of the base manifold. It was proved by Daskalopolos \cite{Dask} over compact Riemann surface, and by R{\aa}de \cite{Ra} in dimensions 2 and 3, that the flow exists for all time and converges. Finite time blow-up phenomena is known to occur in supercritical dimensions ($\dim\geq 5$) \cite{Na}. Work on characterizing the behavior of the flow in supercritcial dimensions has been obtained by Tao-Tian \cite{TT}, and  more recent developments have been made by Petrache-Rivi\'{e}re \cite{PR} in the case of fixed boundary connetions. 4 is the critical dimension, in which singularity formation may occur. Following the analogy with harmonic map heat flow in dimension 2 \cite{Str}, the foundational work of Struwe \cite{Stru} gives a global weak solution for the Yang-Mills flow over closed 4-manifold, without excluding the possibility that point singularities will form in finite time. Later, Schlatter-Struwe-Tahvildar-Zadeh \cite{SST} showed that Yang-Mills flow of $SO(4)$-equivariant connections on an $SU(2)$-bundle over a ball in $\mathbb{R}^4$
admits a smooth solution for all time. This led them to conjecture that long-time existence holds for solutions of Yang-Mills flow in general. Recently, Waldron \cite{Wa} confirmed this conjecture. He proved that finite-time singularities do not occur in 4-dimensional Yang-Mills flow, which is very different from the 2-dimensional harmonic map heat flow \cite{CDY}.

The study of Yang-Mills-Higgs flow has aroused a lot attention in the new century (see \cite{Af,FH,HF,H1,HT1,HT2,LZ0,LS,SW,Tr,Wi2,ZZZ,ZW} and so on). In spite of the work of Waldron \cite{Wa}, it is natural to ask: do the finite-time singularities occur in 4-dimensional Yang-Mills-Higgs flow? We can not solve this question yet. However, we can solve it if the order of derivatives is strictly greater than $1$.

In the following, we will introduce the higher order Yang-Mills-Higgs flow, which will be called as Yang-Mills-Higgs $k$-flow.

For each $k\in \mathbb{N}\cup \{0\}$, the Yang-Mills-Higgs $k$-functional (or Yang-Mills-Higgs $k$-energy) is defined through a connection $\nabla$ and a section $u$ of a vector bundle $E$
\begin{equation} \label{YMHKF}
\mathcal{YMH}_k(\nabla,u)=\frac{1}{2}\int_M\Big[|\nabla^{(k)}F_{\nabla}|^2+|\nabla^{(k+1)}u|^2\Big]d\mathrm{vol}_g.
\end{equation}
When $k=0$, \eqref{YMHKF} is nothing but the Yang-Mills-Higgs functional with vanishing Higgs self-interaction \protect{\cite[page 4]{JT}}.

The Yang-Mills-Higgs $k$-system, i.e. the corresponding Euler-Lagrange equations  of \eqref{YMHKF}, is
\begin{equation} \label{EL}
\begin{cases}
(-1)^kD^*_{\nabla}\Delta_{\nabla}^{(k)}F_{\nabla}+\sum\limits_{v=0}^{2k-1}P_1^{(v)}[F_{\nabla}]+P_2^{(2k-1)}[F_{\nabla}]+\sum\limits_{i=0}^k\nabla^{*(i)}(\nabla^{(k+1)}u\ast \nabla^{(k-i)}u)=0,\\
\nabla^{*(k+1)}\nabla^{(k+1)}u=0,
\end{cases}
\end{equation}
where $\Delta_{\nabla}^{(k)}$ denotes $k$ iterations of the Bochner Laplacian $-\nabla^*\nabla$, and the notation $P$ is defined in \eqref{P}.

A family of pairs $(\nabla_t,u_t)$ is said to be a solution to the Yang-Mills-Higgs $k$-flow if it satisfies the following system
\begin{equation} \label{YMHKS}
\begin{cases}
\frac{\partial \nabla_t}{\partial t}&=(-1)^{(k+1)}D^*_{\nabla_t}\Delta_{\nabla_t}^{(k)}F_{\nabla_t}+\sum\limits_{v=0}^{2k-1}P_1^{(v)}[F_{\nabla_t}]\\
&~~+P_2^{(2k-1)}[F_{\nabla_t}]+\sum\limits_{i=0}^k\nabla_t^{*(i)}(\nabla_t^{(k+1)}u_t\ast \nabla_t^{(k-i)}u_t),\\
\frac{\partial u_t}{\partial t}&=-\nabla^{*(k+1)}_t\nabla_t^{(k+1)}u_t.
\end{cases}
\end{equation}
When $k=0$, the flow \eqref{YMHKS} is a Yang-Mills-Higgs flow \cite{H}.

Now, we state our main result in this paper.

\begin{theorem} \label{mth}
Let $(E,h)$ be a vector bundle over a closed Riemannian $4$-manifold $(M,g)$. Assume the integer $k> 1$, then there exists a unique smooth solution $(\nabla_t,u_t)$ to the Yang-Mills-Higgs $k$-flow \eqref{YMHKS} in $M\times [0,+\infty)$ with smooth initial value $(\nabla_0,u_0)$.
\end{theorem}

\begin{remark}
To prove the long-time existence of the Yang-Mills flow ($k=0$), coupled with an extra structure (Higgs field \cite{H1} or spinor filed \cite{H2}), one powerful tool is the maximal principle. One can obtain a $C^0$ bound of the Higgs field (or spinor field) immediately. This brings us a lot of convenience in the analysis. When $k>0$, the order of $\nabla_t^{*(k+1)}\nabla_t^{(k+1)}$ is more than 2, and the maximal principle fails. We will show that the $L^2$-bound of the Higgs field is enough for energy estimate in $4$ dimension.
\end{remark}

It is not surprising to consider such higher order flow. Just recently, in \cite{Ke}, Kelleher studied higher order Yang-Mills flow (with vanishing Higgs field) and in \cite{Sa}, Saratchandran studied higher order Seiberg-Witten flow (flow of connections coupled with spinor fields). The study of higher order flow also has a long history. In De Giorgi's program to approximate singular geometric flows with sequences of smooth ones, he \cite{DG} conjectured that any compact hypersurfaces in Euclidean space, evolving by the gradient flow of certain functionals with sufficiently high derivatives does not hit singularities. Similar to ones proposed by De Giorgi, Mantegazza \cite{Ma} studied higher order generalisations of the mean curvature flow and proved that the flow do not hit singularities provided the order of the derivatives is sufficiently large. Just very recently, Jia-Wang \cite{JW} extended some results due to Mantegazza to a more general ambient manifold. Actually, there have been many other important works on higher order flow, such as Escher-Mayer-Simonett \cite{EMS} and Wheeler \cite{Wh} for surface diffusion flow, Kuwert-Sch\"{a}tzle \cite{KS0,KS}  and Simonett \cite{Simo} for Willmore flow of surfaces, Streets \cite{Stre} for a certain flow of Riemannian curatures, Bahuaud-Helliwell \cite{BH} and Kotschwar \cite{Ko} for a certain flow of Riemannian metrics, Novaga-Okabe \cite{NO} for steepest descent flow and so on.

Now we will outline the structure of this paper. In Section \ref{pre}, we will give some basic notations. In Section \ref{tho}, we will derive the Euler-Lagrange equations for the Yang-Mills-Higgs $k$-functional and prove the local existence of the flow. In Section \ref{se}, we will obtain $L^2$ derivative estimates of Bernstein-Bando-Shi type and use these to derive a basic obstruction to long time existence. In Section \ref{ba}, we will address the blow-up analysis, which can be used to derive an $L^{\infty}$ bound from $L^p$-bound. In Section \ref{ee}, we will prove that both the Yang-Mills-Higgs energy and the Yang-Mills-Higgs $k$-energy are bounded along the flow in 4 dimension. In Section \ref{pot}, we complete the proof of Theorem \ref{mth}. In Section \ref{cp}, we will show that the long time existence of Yang-Mills-Higgs 1-flow in dimension $4$ is obstructed by the possibility of concentration of the curvature in smaller and smaller balls. In Section \ref{hsi}, we show that provided $\dim(M)<2(k+1)$, the associated negative gradient flow of the Yang-Mills-Higgs $k$-functional with Higgs self-interaction admits long time existence with smooth initial data.

\vskip 1cm

\section{Preliminary} \label{pre}

To meet the requirements in the next sections, here, in this short section the setup and notation are briefly presented. We will use some of Kelleher's notations in \cite{Ke} and Saratchandran's in \cite{Sa}. For a more concentrate elements about Yang-Mills theory, we refer to Donaldson-Kronheimer \cite{DK} and Jost's \cite{Jost} books and references therein.

Let $(E,h)$ be a vector bundle over a smooth closed manifold $(M,g)$ of real dimension $n$. The set of all smooth unitary connections on $E$ will be denoted by $\mathcal{A}_E$. For a given connection $\nabla\in \mathcal{A}_E$, it can be extended to other tensor bundles by coupling with the corresponding Levi-Civita connection $\nabla_M$ on $(M,g)$.

Let $D_{\nabla}$ be the exterior derivative, or skew symmetrization of $\nabla$. The curvature tensor of $E$ is denoted by
 $$F_{\nabla}=D_{\nabla}\circ D_{\nabla}.$$

We set $\nabla^{*},D_{\nabla}^{*}$ to be the formal $L^2$-adjoint of $\nabla,D_{\nabla}$, respectively. The Bochner and Hodge Laplacians are given respectively by
 $$\Delta_{\nabla}=-\nabla^{*}\nabla, \quad \Delta_{D_{\nabla}}=D_{\nabla}D^*_{\nabla}+D^*_{\nabla}D_{\nabla}.$$

Let $\xi,\eta$ be $p$-forms valued in $E$ or $\mathrm{End}(E)$. Let $\xi\ast \eta$ denote any mulitilinear form obtained from a tensor product $\xi\otimes\eta$ in s universal bilinear way. That is to say, $\xi\ast \eta$ is obtained by starting with $\xi\otimes\eta$, taking any linear combination of this tensor, taking any number of metric contractions w.r.t $g$ or $h$, and switching any number of factors in the product. We then have $$|\xi\ast\eta|\leq C|\xi||\eta|.$$

Denote by
$$\nabla^{(i)}=\underbrace{\nabla\cdots\nabla}_{i\ \mathrm{times}}.$$

We will also use the $P$ notation, as introduced in \cite{KS}. Given a tensor $\xi$, we denote by
\begin{equation} \label{P}
P_v^{(k)}[\xi]:=\sum_{w_1+\cdots+w_v=k}(\nabla^{(w_1)}\xi)\ast\cdots\ast(\nabla^{(w_v)}\xi)\ast T,
\end{equation}
where $k,v\in \mathbb{N}$  and $T$ is a generic background tensor dependent only on $g$.

We will collect some lemmas appearing in \cite{Ke,Sa}. During the study of the higher order flow, there will be times when we need to switch derivatives, leading to need the following lemmas.

\subsection{Commutation formulas for connections}

\begin{lemma} [Weitzenb\"{o}ck forumula] \label{wf}
Let $(E,h)$ be a Hermitian vector bundle over a Riemannian manifold $(M,g)$, with compatiable metric connection $\nabla$. Let $\Delta_{D_{\nabla}}=D_{\nabla}D^*_{\nabla}+D^*_{\nabla}D_{\nabla}$ denote the Hodge Laplacian, and $\Delta_{\nabla}=-\nabla^*\nabla$ denote the Bochner Laplacian. For $\phi\in \Omega^p(M;E)$, we have
$$\Delta_{D_{\nabla}}\phi=-\Delta_{\nabla}\phi+(Rm+F_{\nabla})\ast \phi,$$
where $Rm$ denotes Riemannian curvature of $g$.
\end{lemma}
\begin{lemma} \label{a}
Let $(E,h)$ be a Hermitian vector bundle over a Riemannian manifold $(M,g)$, with compatiable metric connection $\nabla$. Let $\phi$ be a section of $E$, we have
\begin{equation*}
\begin{split}
\nabla_{i_k}\nabla_{i_{k-1}}\cdots \nabla_{i_1}\nabla_{j_1}\nabla_{j_2}\cdots\nabla_{j_k}\phi
&=\nabla_{i_k}\nabla_{j_k}\cdots \nabla_{i_1}\nabla_{j_1}\phi\\
&\quad +\sum_{l=0}^{2k-2}\Big[(\nabla_M^{(l)}Rm+\nabla^{(l)}F_{\nabla})\ast \nabla^{(2k-2-l)}\phi\Big].
\end{split}
\end{equation*}
\end{lemma}

\begin{lemma} \label{a0}
Let $(E,h)$ be a Hermitian vector bundle over a Riemannian manifold $(M,g)$, with compatiable metric connection $\nabla$. Let $\phi$ be a section of $E$, we have
$$\nabla^{(n)}\Delta^{(k)}_{\nabla}\phi=\Delta_{\nabla}^{(k)}\nabla^{(n)}\phi+\sum_{j=0}^{2k+n-2}
\Big[(\nabla^{(j)}_MRm+\nabla^{(j)}F_{\nabla})\ast \nabla^{(2k+n-j-2)}\phi\Big].$$
\end{lemma}

\begin{lemma} \label{a1}
Let $(E,h)$ be a Hermitian vector bundle over a Riemannian manifold $(M,g)$, with compatiable metric connection $\nabla$. Let $\xi$ and $\zeta$ are sections of $E$, then for $k\in \mathbb{N}$ we have
\begin{equation*}
\begin{split}
\int_M\langle \nabla^{(k)}\xi,\nabla^{(k)}\zeta\rangle d\mathrm{vol}_g
&=\int_M(-1)^k\langle \xi,\Delta_{\nabla}^{(k)}\zeta\rangle d\mathrm{vol}_g\\
&\quad +\int_M\langle \xi,\sum_{v=0}^{2k-2}\Big((\nabla_M^{(v)}Rm+\nabla^{(v)}F_{\nabla})\ast \nabla^{(2k-2-v)}\zeta\Big) d\mathrm{vol}_g.
\end{split}
\end{equation*}
\end{lemma}
The following interpolation results will be used in Section 4, when proving local derivative estimates.

\subsection{Interpolation inequalities}

\begin{lemma}   (\protect{\cite[Lemma 5.3]{Ke}}, analogue of \protect{\cite[Corollary 5.5]{KS}}) \label{i}
Let $(E,h)$ be a Hermitian vector bundle over a Riemannian manifold $(M,g)$, with connection $\nabla$. Let $\phi$ be a section of $E$, and $\gamma$ a bump function on $M$. For $k\in \mathbb{N}$, if $1\leq i_1,\cdots,i_r\leq k$, $i_1+i_2\cdots+ i_r=2k$ and $s\geq 2k$, we have
\begin{equation*}
\begin{split}
&\int_M   \gamma^s \nabla^{(i_1)}\phi\ast \cdots \ast \nabla^{(i_r)}\phi d\mathrm{vol}_g\\
&\quad \leq C(\dim (M),\mathrm{rk}(E),k,r,s,g,h,\gamma)\|\phi\|^{r-2}_{\infty}\Big(\int_M|\nabla^{(k)}\phi|^2\gamma^sd\mathrm{vol}_g+\|\phi\|^2_{L^2,\gamma>0}\Big),
\end{split}
\end{equation*}
where the subscript $\gamma>0$ means $\{x\in M| \gamma(x)>0\}$.
\end{lemma}

\begin{lemma} (\cite[Corollary 5.2]{Ke}) \label{ii}
Let $(E,h)$ be a Hermitian vector bundle over a Riemannian manifold $(M,g)$, with connection $\nabla$, and $\gamma$ a bump function on $M$. For
$2\leq p<+\infty$, $l\in \mathbb{N}$, $s\geq lp$, there exists $C(\varepsilon^{-1})=C(\varepsilon^{-1}, \dim (M),\mathrm{rk} (E),p,l,s,g,h,\gamma)\in \mathbb{R}_{>0}$ such that for a section $\phi$ of $E$, we have
\begin{equation*}
\|\gamma^{s/p}\nabla^{(l)}\phi\|_{L^p}\leq \varepsilon \|\gamma^{(s+jp)/p}\nabla^{(l+j)}\phi\|_{L^p}+C(\varepsilon^{-1})\|\phi\|_{L^p,\gamma>0}.
\end{equation*}
For $p=2$ and some $K\geq 1$, we have
$$K\|\gamma^{s/2}\nabla^{(l)}\phi\|^2_{L^2}\leq \varepsilon \|\gamma^{(s+2j)/2}\nabla^{(l+j)}\phi\|^2_{L^2}
+C(\varepsilon^{-1})K^2\|\phi\|^2_{L^2,\gamma>0}.$$
\end{lemma}

\vskip 1cm

\section{The higher order Yang-Mills-Higgs flow} \label{tho}

We first computer the Euler-Lagrange equations of the Yang-Mills-Higgs $k$-functional to determine the corresponding Yang-Mills-Higgs $k$-flow. We then prove the local existence of this flow.
\begin{lemma}
The Euler-Lagrange equations associated to the Yang-Mills-Higgs $k$-functional \eqref{YMHKF} are given by \eqref{EL}.
\end{lemma}
\begin{proof}
Let $u_t$ be a path of Higgs fields, with initial value $u_0=u$. Then
\begin{equation} \label{vh1}
\frac{\partial}{\partial t}\Big|_{t=0}\frac{1}{2}\int_M\langle \nabla^{(k+1)}u_t,\nabla^{(k+1)}u_t\rangle=\int_M \langle \frac{\partial u_t}{\partial t},\nabla^{*(k+1)}\nabla^{(k+1)}u_t\rangle\Big|_{t=0}.
\end{equation}

Next, let $\nabla_t$ be a path of connections, with initial value $\nabla_0=\nabla$. Then
\begin{equation} \label{vc1}
\begin{split}
\frac{\partial}{\partial t}\Big|_{t=0}\frac{1}{2}\int_M\langle \nabla_t^{(k+1)}u,\nabla^{(k+1)}_tu\rangle
&=\int_M \langle \frac{\partial(\nabla_t^{(k+1)}u)}{\partial t},\nabla_t^{(k+1)}u\rangle\Big|_{t=0}\\
&=\int_M \langle \sum_{i=0}^k\nabla_t^{(i)}\frac{\partial \nabla_t}{\partial t}\ast \nabla_t^{(k-i)}u,\nabla_t^{(k+1)}u\rangle\Big|_{t=0}\\
&=\int_M \langle \frac{\partial \nabla_t}{\partial t},\sum_{i=0}^k\nabla_t^{*(i)} (\nabla_t^{(k+1)}u\ast \nabla_t^{(k-i)}u)\rangle\Big|_{t=0},
\end{split}
\end{equation}
where we used the following variation formula, which can be proved by induction on $k$:
\begin{equation} \label{sa}
\frac{\partial}{\partial t}\Big(\nabla^{(k+1)}_tu_t\Big)=\nabla^{(k+1)}_t\frac{\partial u_t}{\partial t}+\sum_{i=0}^{k}\Big(\nabla_t^{(i)}\frac{\partial \nabla_t}{\partial t}\Big)\ast \Big(\nabla_t^{(k-i)}u_t\Big).
\end{equation}
At last, we compute
\begin{equation} \label{vc2}
\begin{split}
&\frac{\partial}{\partial t}\Big|_{t=0}\frac{1}{2}\int_M\langle \nabla_t^{(k)}F_{ \nabla_t},\nabla^{(k)}_tF_{\nabla_t}\rangle\\
&=\int_M \langle \frac{\partial(\nabla^{(k)}_tF_{\nabla_t})}{\partial t}, \nabla^{(k)}_tF_{\nabla_t}\rangle\Big|_{t=0}\\
&=\int_M \langle \nabla^{(k)}_t\frac{\partial F_{\nabla_t}}{\partial t}+\sum_{i=0}^{k-1}\nabla_t^{(i)}\frac{\partial \nabla_t}{\partial t}\ast \nabla_t^{(k-i-1)}F_{\nabla_t},\nabla^{(k)}_tF_{\nabla_t}\rangle\Big|_{t=0}\\
&=\int_M \langle \frac{\partial \nabla_t}{\partial t}, (-1)^kD^*_{\nabla_t}\Delta_{\nabla_t}^{(k)}F_{\nabla_t}+\sum_{v=0}^{2k-1}P_1^{(v)}[F_{\nabla_t}]+P_2^{(2k-1)}[F_{\nabla_t}]\rangle \Big|_{t=0},
\end{split}
\end{equation}
where we used (\protect {\cite[Corollary 2.2]{Ke}})
\begin{equation} \label{ke}
\frac{\partial}{\partial t}\Big(\nabla^{(k)}_tF_{\nabla_t}\Big)=\nabla^{(k)}_t\frac{\partial F_{\nabla_t}}{\partial t}+\sum_{i=0}^{k-1}\nabla_t^{(i)}\frac{\partial \nabla_t}{\partial t}\ast \nabla_t^{(k-i-1)}F_{\nabla_t},
\end{equation}
$$\frac{\partial F_{\nabla_t}}{\partial t}=D_{\nabla_t}\frac{\partial \nabla_t}{\partial t},$$
and Lemma \ref{a1}.

Hence we prove the lemma by combining \eqref{vh1},\eqref{vc1} and \eqref{vc2}.

\end{proof}

Given 1-parameter pairs $(\nabla_t,u_t)$, we can define Yang-Mills-Higgs $k$-flow by \eqref{YMHKS}. Then we will use De Turck's trick to establish the local existence of the Yang-Mills-Higgs $k$-flow. We refer to \cite{Ke} for more details. The proof is standard, we will outline the procedures.

\begin{theorem} \label{le}
Let $(E,h)$ be a vector bundle over a closed Riemannian manifold $(M,g)$. There exists a unique smooth solution $(\nabla_t,u_t)$ to the Yang-Mills-Higgs $k$-flow in $M\times [0,\epsilon)$ with smooth initial value $(\nabla(0),u(0))$.
\end{theorem}

\begin{proof} ({\bf Local existence}) We first introduce 1-parameter $(\widetilde{\nabla}_t,\widetilde{u_t})$ satisfying the following system
\begin{equation} \label{ps}
\begin{cases}
\frac{\partial \widetilde{\nabla}_t}{\partial t}=(-1)^{k+1}D^*_{\widetilde{\nabla}}\Delta^{(k)}_{\widetilde{\nabla}_t}F_{\widetilde{\nabla}_t}
+(-1)^kD_{\widetilde{\nabla}_t}\Delta_{\widetilde{\nabla}_t}^{(k)}D^{*}_{\widetilde{\nabla}_t}(\widetilde{\nabla}_t-\nabla(0))
+\sum\limits_{v=0}^{2k-1}P_1^{(v)}[F_{\widetilde{\nabla}_t}]\\
\quad \quad +P_2^{(2k-1)}[F_{\widetilde{\nabla}_t}]+\sum\limits_{i=0}^k\widetilde{\nabla}_t^{*(i)}(\widetilde{\nabla}_t^{(k+1)}u_t\ast \widetilde{\nabla}_t^{(k-i)}u_t),\\
\frac{\partial \widetilde{u}_t}{\partial t}=-\widetilde{\nabla}_t^{*(k+1)}\widetilde{\nabla}_t^{(k+1)}\widetilde{u}_t-(-1)^k\Big(\Delta_{\widetilde{\nabla}_t}^{(k)}D^{*}_{\widetilde{\nabla}_t}(\widetilde{\nabla}_t-\nabla(0))\Big)\widetilde{u}_t,\\
\widetilde{\nabla}(0)=\nabla(0),\\
 \widetilde{u}(0)=u(0).
\end{cases}
\end{equation}

From \protect {\cite[Lemma 3.2]{Ke}}, the operator $\Phi_k(\cdot,\nabla(0))$, given by
$$ \Phi_k(\widetilde{\nabla}_t,\nabla(0))=
(-1)^{k+1}D^*_{\widetilde{\nabla}_t}\Delta^{(k)}_{\widetilde{\nabla}_t}F_{\widetilde{\nabla}_t}
+(-1)^kD_{\widetilde{\nabla}_t}\Delta_{\widetilde{\nabla}_t}^{(k)}D^{*}_{\widetilde{\nabla}_t}(\widetilde{\nabla}_t-\nabla(0))
$$
is elliptic.
Using Lemma \ref{a}, we have
$$\widetilde{\nabla}_t^{*(k+1)}\widetilde{\nabla}_t^{(k+1)}\widetilde{u}_t=(-1)^{k+1}\Delta^{(k+1)}_{\widetilde{\nabla}_t}\widetilde{u}_t+\sum_{i=0}^{2k}(\nabla_M^{(i)}Rm+\widetilde{\nabla}^{(i)}F_{\widetilde{\nabla}_t})
\ast\widetilde{\nabla}^{(2k-i)}_t\widetilde{u}_t.$$
From \protect {\cite[Lemma 3.5]{Ke}}, we have
$$\Big(\Delta_{\widetilde{\nabla}_t}^{(k)}D^{*}_{\widetilde{\nabla}_t}(\widetilde{\nabla}_t-\nabla(0))\Big)\widetilde{u}_t=-\Delta_{\widetilde{\nabla}_t}^{(k+1)}\widetilde{u}_t+\alpha(\widetilde{\nabla}_t,\widetilde{u}_t),$$
where $\alpha(\widetilde{\nabla}_t,\widetilde{u}_t)$ is lower order than $\Delta_{\widetilde{\nabla}_t}^{(k+1)}\widetilde{u}_t$.
Hence ellipticity of the highest order term in the system \eqref{ps} follows. Therefore, the system \eqref{ps} is parabolic and has short time existence.

Define a gauge $g(t)$ as
\begin{equation*} \label{gauge}
\begin{cases}
\frac{\partial g(t)}{\partial t}=(-1)^{k+1}\Delta_{\widetilde{\nabla}_t}^{(k)}D^{*}_{\widetilde{\nabla}_t}(\widetilde{\nabla}_t-\nabla(0))g(t)\\
g(0)=\mathrm{id}.
\end{cases}
\end{equation*}
One can check that $(g(t)^*\widetilde{\nabla}_t,g(t)^*\widetilde{u}_t)$ satisfies the Yang-Mills-Higgs $k$-flow \eqref{YMHKS} with initial condition $(g(0)^*\widetilde{\nabla}_0,g(0)^*\widetilde{u}_0)=(\nabla_0,u_0)$. This proves the short time existence.

({\bf Uniqueness}) If we have two solutions to the Yang-Mills-Higgs $k$-flow \eqref{YMHKS}, $(\nabla_1(t),u_1(t))$ and $(\nabla_2(t),u_2(t))$, with the same initial value $(\nabla(0),u(0))$. Then we can define two gauges $g_1$ and $g_2$ that satisfy the above gauge transformation equations, with $\nabla_1$ and $\nabla_2$ respectively. We then find that $((g_1^{-1})^*\nabla_1,(g_1^{-1})^*u_1)$ and $((g_2^{-1})^*\nabla_2,(g_2^{-1})^*u_2)$ both solve the parabolic system \eqref{ps} with the same initial value $(\nabla(0),u(0))$. Uniqueness of this system implies that
$$((g_1^{-1})^*\nabla_1,(g_1^{-1})^*u_1)=((g_2^{-1})^*\nabla_2,(g_2^{-1})^*u_2),$$
 which means
 $$(\nabla_1,u_1)=((g_2^{-1}g_1)^*\nabla_2,(g_2^{-1}g_1)^*u_2).$$
 Define a new gauge $g_3=g_2^{-1}g_1$, a direct calculation yields
\begin{equation*}
\begin{cases}
\frac{\partial g_3}{\partial t}=g_3(-1)^{k+1}\Delta^{(k)}_{g_3^*\nabla_2}D^*_{g_3^*\nabla_2}(g_3^{*}\nabla_2-\nabla(0))-(-1)^{k+1}\Delta^{(k)}_{\nabla_2}D^*_{\nabla_2}(\nabla_2-\nabla(0))g_3,\\
g_3(0)=\mathrm{id}.
\end{cases}
\end{equation*}
Clearly, $\mathrm{id}$ is a solution to the above ODE. By uniqueness of ODE, $g_3(t)=\mathrm{id}$.

\end{proof}

\vskip 1cm

\section{Smoothing estimates} \label{se}

In this section, our goal is to obtain derivative estimates of $F_{\nabla_t}$ and $u_t$. To accomplish this we first compute necessary evolution equations.

\subsection{Evolution equations}

\begin{lemma}
Suppose $(\nabla_t,u_t)$ is a solution to the Yang-Mills-Higgs $k$-flow \eqref{YMHKS} defined on $M\times [0,T)$. Then
\begin{equation} \label{ef1}
\frac{\partial F_{\nabla_t}}{\partial t}=(-1)^{k}\Delta_{\nabla_t}^{(k+1)}F_{\nabla_t}+\sum_{v=0}^{2k}P^{(v)}_1[F_{\nabla_t}]+P_2^{(2k)}[F_{\nabla_t}]
+\sum_{i=0}^kD_{\nabla_t}\nabla_t^{*(i)}(\nabla_t^{(k+1)}u_t\ast \nabla_t^{(k-i)}u_t),
\end{equation}
and for $l\in\mathbb{N}$,
\begin{equation} \label{ef2}
\begin{split}
\frac{\partial }{\partial t}[\nabla_t^{(l)}F_{\nabla_t}]
&=(-1)^{k}\Delta_t^{(k+1)}\nabla_t^{(l)}F_{\nabla_t}+\sum_{v=0}^{2k+l}\Big(P^{(v)}_1[F_{\nabla_t}]+P_2^{(v)}[F_{\nabla_t}]\Big)\\
&\quad +P_3^{(2k+l-2)}[F_{\nabla_t}]+\sum_{i=0}^k\nabla_t^{(l)}D_{\nabla_t}\nabla_t^{*(i)}(\nabla_t^{(k+1)}u_t\ast \nabla_t^{(k-i)}u_t)\\
&\quad +\sum_{j=0}^{l-1}\sum_{i=0}^k\Big[\nabla^{(j)}_t\Big(\nabla^{*(i)}_t(\nabla^{(k+1)}_tu_t\ast \nabla_t^{(k-i)}u_t)\Big)\Big]\ast \nabla_t^{(l-j-1)}F_{\nabla_t},
\end{split}
\end{equation}

\end{lemma}

\begin{proof}
From \eqref{YMHKS} and Weitzenb\"{o}ck formula (Lemma \ref{wf}), we have
\begin{equation}
\begin{split}
\frac{\partial F_{\nabla_t}}{\partial t}
&=D_{\nabla_t}\frac{\partial \nabla_t}{\partial t}=(-1)^{k+1}D_{\nabla_t}D_{\nabla_t}^{*}\Delta_{\nabla_t}^{(k)}F_{\nabla_t}+\sum_{v=0}^{2k}P_1^{(v)}[F_{\nabla_t}]\\
&\quad +P_2^{(2k)}[F_{\nabla_t}]
+\sum_{i=0}^kD_{\nabla_t}\nabla_t^{*(i)}\Big(\nabla_t^{(k+1)}u_t\ast \nabla_t^{(k-i)}u_t\Big)\\
&=(-1)^{k+1}\Delta_{D_{\nabla_t}}\Delta_{\nabla_t}^{(k)}F_{\nabla_t}+\sum_{v=0}^{2k}P_1^{(v)}[F_{\nabla_t}]+P_2^{(2k)}[F_{\nabla_t}]\\
&\quad +\sum_{i=0}^kD_{\nabla_t}\nabla_t^{*(i)}\Big(\nabla_t^{(k+1)}u_t\ast \nabla_t^{(k-i)}u_t\Big)\\
&=(-1)^{k}\Delta_{\nabla_t}^{(k+1)}F_{\nabla_t}+(Rm+F_{\nabla_t})\ast(\Delta^{(k)}_{\nabla_t}F_{\nabla_t})+\sum_{v=0}^{2k}P_1^{(v)}[F_{\nabla_t}]\\
&\quad +P_2^{(2k)}[F_{\nabla_t}]+\sum_{i=0}^kD_{\nabla_t}\nabla_t^{*(i)}\Big(\nabla_t^{(k+1)}u_t\ast \nabla_t^{(k-i)}u_t\Big),
\end{split}
\end{equation}
which implies \eqref{ef1}.

From \eqref{ke} we have
\begin{equation}
\frac{\partial}{\partial t}[\nabla_t^{(l)}F_{\nabla_t}]
=\Big[\nabla_t^{(l)}D_{\nabla_t}\frac{\partial \nabla_t}{\partial t}\Big]_{T_1}
+\Big[\sum_{j=0}^{l-1}\Big(\nabla_t^{(j)}\frac{\partial \nabla_t}{\partial t}\ast \nabla^{(l-j-1)}_tF_{\nabla_t}\Big)\Big]_{T_2}.
\end{equation}

We manipulate $T_1$ first. Using the Weitzenb\"{o}ck formula (Lemma \ref{wf}) and Lemma \ref{a0} yields
\begin{equation*}
\begin{split}
T_1&=(-1)^k\nabla_t^{(l)}\Delta_{\nabla_t}^{(k+1)}F_{\nabla_t}+\nabla_t^{(l)}[(Rm+F_{\nabla_t})\ast \Delta^{(k)}_{\nabla_t}F_{\nabla_t}]\\
&\quad +\sum_{v=0}^{2k+l}P_1^{(v)}[F_{\nabla_t}]+P_2^{(2k+l)}[F_{\nabla_t}]
+\sum_{i=0}^k\nabla_t^{(l)}D_{\nabla_t}\nabla_t^{*(i)}\Big(\nabla_t^{(k+1)}u_t\ast \nabla_t^{(k-i)}u_t\Big)\\
&=(-1)^{k}\Delta_{\nabla_t}^{(k+1)}\nabla_t^{(l)}F_{\nabla_t}+\sum_{v=0}^{2k+l}P_1^{(v)}[F_{\nabla_t}]
+P_2^{(2k+l)}[F_{\nabla_t}]\\
&\quad +\sum_{i=0}^k\nabla_t^{(l)}D_{\nabla_t}\nabla_t^{*(i)}\Big(\nabla_t^{(k+1)}u_t\ast \nabla_t^{(k-i)}u_t\Big).
\end{split}
\end{equation*}

Next, we manipulate $T_2$.
\begin{equation*}
\begin{split}
T_2&=
\sum_{j=0}^{l-1}\Big[
\nabla_t^{(j)}
\Big(
(-1)^{k+1}D^*_{\nabla_t}\Delta^{(k)}_{\nabla_t}F_{\nabla_t}+\sum_{v=0}^{2k-1}P_1^{(v)}[F_{\nabla_t}]+P_2^{(2k-1)}[F_{\nabla_t}]\\
&\quad +\sum_{i=0}^k\nabla_t^{*(i)}(\nabla_t^{(k+1)}u_t\ast \nabla_t^{(k-i)}u_t)\Big)\ast\nabla_t^{(l-j-1)}F_{\nabla_t}
\Big]\\
&=P_2^{(2k+l)}[F_{\nabla_t}]+\sum_{v=0}^{2k+l-2}P_2^{(v)}[F_{\nabla_t}]+P_3^{(2k+l-2)}[F_{\nabla_t}]\\
&\quad +\sum_{j=0}^{l-1}\sum_{i=0}^k\Big[\nabla^{(j)}_t\Big(\nabla^{*(i)}_t(\nabla_t^{(k+1)}u_t\ast \nabla^{(k-i)}_tu_t)\Big)\Big]\ast \nabla_t^{(l-j-1)}F_{\nabla_t}.
\end{split}
\end{equation*}

Combining $T_1$ and $T_2$ yields \eqref{ef2}.

\end{proof}

\begin{lemma} \label{eh}
Suppose $(\nabla_t,u_t)$ is a solution to the Yang-Mills-Higgs $k$-flow \eqref{YMHKS} defined on $M\times [0,T)$. Then
\begin{equation}
\begin{split}
\frac{\partial }{\partial t}[\nabla_t^{(l)}u_t]
&=(-1)^{k}\Delta_{\nabla_t}^{(k+1)}\nabla_{t}^{(l)}u_t+\sum_{j=0}^{2k+l}(\nabla_M^{(j)}Rm+\nabla_t^{(j)}F_{\nabla_t})\ast\nabla_t^{(2k+l-j)}u_t\\
&\quad +\sum_{j=0}^{l-1}\nabla_t^{(j)}D_{\nabla_t}^*\Delta_{\nabla_t}^{(k)}F_{\nabla_t}\ast \nabla_{t}^{(l-j-1)}u_t
+\sum_{j=0}^{l-1}\sum_{v=0}^{2k+j-1}P_1^{(v)}[F_{\nabla_t}]\ast \nabla_t^{(l-j-1)}u_t\\
&\quad +\sum_{j=0}^{l-1}P_2^{(2k+j-1)}[F_{\nabla_t}]\ast \nabla_t^{(l-j-1)}u_t\\
&\quad +\sum_{j=0}^{l-1}\sum_{i=0}^k\Big[\nabla_t^{(j)}\nabla_t^{*(i)}(\nabla_t^{(k+1)}u_t\ast \nabla_t^{(k-i)}u_t)\Big]\ast \nabla_t^{(l-j-1)}u_t.
\end{split}
\end{equation}
\end{lemma}

\begin{proof}
From \eqref{sa} and \eqref{YMHKS}, we have
\begin{equation*}
\begin{split}
\frac{\partial}{\partial t}[\nabla_t^{(l)}u_t]
&=\nabla_t^{(l)}\frac{\partial u_t}{\partial t}+\sum_{j=0}^{l-1}\nabla_t^{(j)}\frac{\partial \nabla_t}{\partial t}\ast \nabla_t^{(l-j-1)}u_t\\
&=\nabla_t^{(l)}(-\nabla_t^{*(k+1)}\nabla_t^{(k+1)}u_t)+\sum_{j=0}^{l-1}\nabla_t^{(j)}\Big[(-1)^{k+1}
D^*_{\nabla_t}\Delta_{\nabla_t}^{(k)}F_{\nabla_t}+\sum\limits_{v=0}^{2k-1}P_1^{(v)}[F_{\nabla_t}]\\
&\quad +P_2^{(2k-1)}[F_{\nabla_t}]+\sum\limits_{i=0}^k\nabla_t^{*(i)}(\nabla_t^{(k+1)}u_t\ast \nabla_t^{(k-i)}u_t)\Big]\ast \nabla_t^{(l-j-1)}u_t.
\end{split}
\end{equation*}
Then using Lemma \ref{a} and Lemma \ref{a0} yields the desired result.

\end{proof}

\subsection{Estimates for derivatives of the Higgs field}

In this subsection, we will prove local $L^2$-derivative estimates for the Higgs field.

The following proposition is a direct consequence of Lemma \ref{eh}.

\begin{proposition} \label{eh1}
Suppose $(\nabla_t,u_t)$ is a solution to the Yang-Mills-Higgs $k$-flow \eqref{YMHKS} defined on $M\times [0,T)$. Then
\begin{equation} \label{hf}
\begin{split}
\frac{\partial}{\partial t}\|\gamma^{s/2}\nabla_t^{(l)}u_t\|_{L^2}^2
&=2(-1)^{k}\int_M \Big\langle \Delta_{\nabla_t}^{(k+1)}\nabla_t^{(l)}u_t,\gamma^s\nabla_t^{(l)}u_t\Big\rangle\\
&\quad +\int_M \sum_{j=0}^{2k+l}\Big\langle (\nabla_M^{(j)}Rm+ \nabla_t^{(j)}F_{\nabla_t})\ast \nabla_t^{(2k+l-j)}u_t,\gamma^s \nabla_t^{(l)}u_t\Big\rangle\\
&\quad +\int_M \sum_{j=0}^{l-1}\Big\langle \nabla_t^{(j)}D_{\nabla_t}^*\Delta_{\nabla_t}^{(k)}F_{\nabla_t}\ast \nabla_t^{(l-j-1)}u_t,\gamma^s \nabla_t^{(l)}u_t\Big\rangle\\
&\quad +\int_M \sum_{j=0}^{l-1}\Big\langle (\sum_{v=0}^{2k+j-1}P_1^{(v)}[F_{\nabla_t}]+P_2^{(2k+j-1)}[F_{\nabla_t}])
\ast \nabla_t^{(l-j-1)}u_t,\gamma^s\nabla_t^{(l)}u_t\Big\rangle\\
&\quad +\int_M \sum_{j=0}^{l-1}\sum_{i=0}^k\Big\langle\Big[ \nabla_t^{(j)}\nabla_t^{*(i)}(\nabla_t^{(k+1)}u_t\ast \nabla_t^{(k-i)}u_t)\Big]\ast \nabla_t^{(l-j-1)}u_t,
\gamma^s \nabla_t^{(l)}u_t\Big\rangle.
\end{split}
\end{equation}
\end{proposition}

We will estimate each term on the right hand side of the above equality. We first introduce the bump function, which is highly necessary in the smooth estimates.

\begin{definition}
[Bump function] Let $\mathfrak{B}:=\{\gamma\in C^{\infty}_c(M):0\leq \gamma\leq 1\}$, that is, the family of bump functions. For $l\in \mathbb{N}$, we denote by
$$J^{(l)}_{\gamma}:=\sum_{j=0}^l\|\nabla^{(j)}\gamma\|_{L^{\infty}(M)}.$$
\end{definition}

We also need the following lemma. This can be proved by integration by parts and by induction method.

\begin{lemma}   [\protect {\cite[Lemma 3.10]{Ke}}] \label{K1}
Let $p,q,r,s\in \mathbb{N}$, $\nabla\in \mathcal{A}_E$ and $\gamma\in \mathfrak{B}$. If $s\in \mathbb{N}\backslash \{1\}$, then
\begin{equation*}
\begin{split}
\int_M(P_1^{(p)}[\phi]\ast P_1^{(q+r)}[\phi])\gamma^sd\mathrm{vol}_g
&\leq \int_M(P_1^{(p+r)}[\phi]\ast P_1^{(q)}[\phi])\gamma^sd\mathrm{vol}_g\\
&+\sum_{j=0}^{r-1}J^{(1)}_{\gamma}\int_M(P_1^{(p+j)}[\phi]\ast P_1^{(q+r-j-1)}[\phi])\gamma^{s-1}d\mathrm{vol}_g,
\end{split}
\end{equation*}
where $\phi$ is in some tensor product of $TM,E$, and their corresponding duals.

\end{lemma}

Now, we are ready to handle the right hand side of \eqref{hf}.

\begin{lemma} \label{se1}
Suppose $(\nabla_t,u_t)$ is a solution to the Yang-Mills-Higgs $k$-flow \eqref{YMHKS} defined on $M\times [0,T)$. Assume  $Q=\max\{1,\sup\limits_{t\in [0,T)} |F_{\nabla_t}|\}$, $K=\max\{1,\sup\limits_{t\in [0,T)}|u_t|\}$, $\gamma$ is bump function. Then for $s\geq2(k+l+1)$, there exist $\lambda\in[1,2)$ and $C:=C(\mathrm{dim} (M),\mathrm{rk} (E),s,k,l,g,h,\gamma)$ such that
\begin{equation*}
\begin{split}
2(-1)^{k}\int_M \Big\langle \Delta_{\nabla_t}^{(k+1)}\nabla_t^{(l)}u_t,\gamma^s\nabla_t^{(l)}u_t\Big\rangle
&\leq -\lambda\|\gamma^{s/2}\nabla_t^{(k+l+1)}u_t\|^2_{L^2}+CQK^2\|u_t\|^2_{L^2,\gamma>0}.
\end{split}
\end{equation*}
\end{lemma}

\begin{proof}
From Lemma \ref{a1}, we have
\begin{equation}
\begin{split}
2(-1)^{k}\int_M \Big\langle \Delta_{\nabla_t}^{(k+1)}\nabla_t^{(l)}u_t,\gamma^s\nabla_t^{(l)}u_t\Big\rangle
&= \Big[-2\int_M \Big\langle \nabla_t^{(k+1)}\nabla_t^{(l)}u_t,\nabla_t^{(k+1)}(\gamma^s\nabla^{(l)}_tu_t)\Big\rangle\Big]_{T_1}\\
&~~+\Big[\int_M\sum_{j=0}^{2k}\Big\langle \nabla_M^{(j)}Rm\ast \nabla_t^{(2k+l-j)}u_t,\gamma^s \nabla^{(l)}_tu_t\Big\rangle\Big]_{T_2}\\
&~~+\Big[\int_M\sum_{j=0}^{2k}\Big\langle \nabla^{(j)}_t F_{\nabla_t}\ast \nabla_t^{(2k+l-j)}u_t,\gamma^s \nabla^{(l)}_tu_t\Big\rangle\Big]_{T_3}.
\end{split}
\end{equation}

We manipulate $T_1$ first. Direct computation yields
\begin{equation*}
\begin{split}
T_1&=-2\|\gamma^{s/2}\nabla^{(k+l+1)}_tu_t\|^2_{L^2}+\int_M\sum_{j=1}^{k+1}\nabla^{(j)}\gamma^s\ast\Big\langle \nabla_t^{(k+l+1)}u_t,\nabla_t^{(k+l+1-j)}u_t\Big\rangle\\
&\leq -2\|\gamma^{s/2}\nabla^{(k+l+1)}_tu_t\|^2_{L^2}+\int_M\sum_{j=1}^{k+1}C|\gamma^{s/2} \nabla_t^{(k+l+1)}u_t||\gamma^{\frac{s-2j}{2}}\nabla_t^{(k+l+1-j)}u_t|\\
&\leq -2\|\gamma^{s/2}\nabla^{(k+l+1)}_tu_t\|^2_{L^2}+C\varepsilon_1\|\gamma^{s/2} \nabla_t^{(k+l+1)}u_t\|^2_{L^2}\\
&\quad +\sum_{j=1}^{k+1}\frac{C}{\varepsilon_1}\|\gamma^{\frac{s-2j}{2}}\nabla_t^{(k+l+1-j)}u_t\|^2_{L^2}\\
&\leq (-2+C(\varepsilon_1+\varepsilon^{-1}\varepsilon_2))\|\gamma^{s/2}\nabla^{(k+l+1)}_tu_t\|^2_{L^2}+C\varepsilon_1^{-1}\varepsilon^{-1}_2K^2\|u_t\|^2_{L^2,\gamma>0},
\end{split}
\end{equation*}
where we used the following identity (\protect {\cite[Lemma 6.2]{Sa}}) in the second inequality
\begin{equation} \label{ide}
\nabla^{(j)}\gamma^s=\sum_{p_1+\cdots+p_j=j}C_{p_1,\cdots,p_j}(\gamma,s)\gamma^{s-j}\nabla^{(p_1)}\gamma\ast\cdots\ast\nabla^{(p_j)}\gamma,
\end{equation}
and used Lemma \ref{ii} in the last inequality.

Next, we manipulate $T_2$. We divide up the summation into cases when $j$ is either odd or even, and apply Lemma \ref{K1},
\begin{equation*}
\begin{split}
T_2&=\int_M\sum_{j:j\in 2\mathbb{N}\cup \{0\}}^{2k}P_2^{(2k+2l-j)}[u_t]\gamma^{s}+\int_M\sum_{j:j\in 2\mathbb{N}-1}^{2k-1}P_2^{(2k+2l-j)}[u_t]\gamma^{s}\\
&\leq \Big[\int_M\sum_{j:j\in 2\mathbb{N}\cup \{0\}}^{2k}P_2^{(2k+2l-j)}[u_t]J_{\gamma}^{(1)}\gamma^{s-1}\Big]_{T_2,Even}\\
&~~+\Big[\int_M\sum_{j:j\in 2\mathbb{N}-1}^{2k-1}P_1^{(\lceil \frac{2k+2l-j}{2}\rceil)}[u_t]\ast P_1^{(\lfloor\frac{2k+2l-j}{2}\rfloor)}[u_t]\gamma^{s}\Big]_{T_2,Odd}.
\end{split}
\end{equation*}
For the even part of $T_2$, apply Lemma \ref{i} and Lemma \ref{ii}, we have
\begin{equation*}
\begin{split}
\int_M P_2^{(2k+2l-j)}[u_t]J_{\gamma}^{(1)}\gamma^{s-1}
&\leq C\Big(\|\gamma^{(s-1)/2}\nabla_t^{(k+l-\frac{j}{2})}u_t\|^2_{L^2}+\|u_t\|^2_{L^2,\gamma>0}\Big)\\
&=C\Big(\|\gamma^{(s-1)/2}\nabla_t^{(k+l+1-1-\frac{j}{2})}u_t\|^2_{L^2}+\|u_t\|^2_{L^2,\gamma>0}\Big)\\
&\leq \varepsilon \|\gamma^{s/2}\nabla_t^{(k+l+1)}u_t\|^2_{L^2}+CK^2\|u_t\|^2_{L^2,\gamma>0}.
\end{split}
\end{equation*}
For the odd part of $T_2$, apply H\"{o}lder inequality, Lemma \ref{i} and Lemma \ref{ii}, we have
\begin{equation*}
\begin{split}
&\int_MP_1^{(\lceil \frac{2k+2l-j}{2}\rceil)}[u_t]\ast P_1^{(\lfloor\frac{2k+2l-j}{2}\rfloor)}[u_t]\gamma^{s}\\
&\leq 2 \int_M P_2^{(2\lceil \frac{2k+2l-j}{2}\rceil)}[u_t]\gamma^{s}
+2\int_M P_2^{(2\lfloor\frac{2k+2l-j}{2}\rfloor)}[u_t]\gamma^{s}\\
&\leq C\Big(\|\gamma^{s/2}\nabla_t^{(\lceil \frac{2k+2l-j}{2}\rceil)}u_t\|^2_{L^2}+\|u_t\|^2_{L^2,\gamma>0}\Big)
+C\Big(\|\gamma^{s/2}\nabla_t^{(\lfloor \frac{2k+2l-j}{2}\rfloor)}u_t\|^2_{L^2}+\|u_t\|^2_{L^2,\gamma>0}\Big)\\
&\leq \varepsilon\|\gamma^{s/2}\nabla_t^{(k+l+1)}u_t\|^2_{L^2}+CK^2\|u_t\|^2_{L^2,\gamma>0}.
\end{split}
\end{equation*}
Therefore we conclude that
$$T_2\leq \varepsilon\|\gamma^{s/2}\nabla_t^{(k+l+1)}u_t\|^2_{L^2}+CK^2\|u_t\|^2_{L^2,\gamma>0}.$$

At last, we will manipulate $T_3$.
\begin{equation*}
\begin{split}
T_3&=\int_M \sum_{j=0}^{2k}\sum_{i=0}^{j}\Big\langle \nabla_t^{(i)}(F_{\nabla_t}\ast \nabla_t^{(2k+l-i)}u_t),\gamma^s\nabla_t^{(l)}u_t\Big\rangle\\
&=\int_M \sum_{j=0}^{2k}\sum_{i=0}^{j}\Big\langle F_{\nabla_t}\ast \nabla_t^{(2k+l-i)}u_t,P_1^{(i)}[\gamma^s\nabla_t^{(l)}u_t]\Big\rangle\\
&=\int_M \sum_{j=0}^{2k}\sum_{i=0}^{j}\Big\langle F_{\nabla_t}\ast \nabla_t^{(2k+l-i)}u_t,\sum_{v=0}^{i}\nabla^{(v)}\gamma^s\ast \nabla_t^{(l+i-v)}u_t\Big\rangle\\
&\leq  CQ\int_M \sum_{v=0}^{2k}\gamma^{s-v} P_2^{(2k+2l-v)}[u_t],\\
\end{split}
\end{equation*}
where we used \eqref{ide}. We divide up the summation into cases when $v$ is either odd or even. Similar to $T_2$, we have
$$T_3\leq \varepsilon\|\gamma^{s/2}\nabla_t^{(k+l+1)}u_t\|^2_{L^2}+CQK^2\|u_t\|^2_{L^2,\gamma>0}.$$

Combining $T_1,T_2$ and $T_3$ we complete the proof. The constraints on $s$ can be easily checked, and we omit here.

\end{proof}

Similar to the proof of Lemma \ref{se1}, we can derive the estimates for the rest terms of \eqref{hf}. Except the term involving $P_2^{(2k+j-1)}[F_{\nabla_t}]$, the rest terms of \eqref{hf} are very similar to the ones appearing in Sarathcandran's paper \protect {\cite[Proposition 6.7]{Sa}}. As for $P_2^{(2k+j-1)}[F_{\nabla_t}]$, we can write it as
$$P_2^{(2k+j-1)}[F_{\nabla_t}]=\sum_{\mu=0}^{2k+j-1}P_1^{(\mu)}[F_{\nabla_t}]\ast P_1^{(2k+j-1-\mu)}[F_{\nabla_t}]=\sum_{\mu=0}^{2k+j-1}\sum_{\nu=0}^{\mu}\nabla_t^{(\nu)}(F_{\nabla_t}\ast \nabla_t^{(2k+j-1-\nu)}[F_{\nabla_t}]).$$ Then integration by parts $\nu$-times, we can obtain a $F_{\nabla_t}$ standing by one side, then ``throw'' it and integration by parts again, and then ``throw'' it again. This makes sense because $F_{\nabla_t}$ is bounded. Therefore, we have the following local $L^2$-derivative estimate for Higgs field.
\begin{proposition} \label{ehf}
Suppose $(\nabla_t,u_t)$ is a solution to the Yang-Mills-Higgs $k$-flow \eqref{YMHKS} defined on $M\times [0,T)$. Assume  $Q=\max\{1,\sup\limits_{t\in [0,T)} |F_{\nabla_t}|\}$, $K=\max\{1,\sup\limits_{t\in [0,T)}|u_t|\}$, $\gamma$ is bump function. Then for $s\geq2(k+l+1)$, there exist $\lambda\in [1,2)$ and $C:=C(\mathrm{dim} (M),\mathrm{rk} (E),s,k,l,g,h,\gamma)$ such that
\begin{equation*}
\begin{split}
\frac{\partial}{\partial t}\|\gamma^{s/2}\nabla_t^{(l)}u_t\|_{L^2}^2
&\leq -\lambda\|\gamma^{s/2}\nabla_t^{(k+l+1)}u_t\|^2_{L^2}+CQ^2K^4\|u_t\|^2_{L^2,\gamma>0}.
\end{split}
\end{equation*}
\end{proposition}

\subsection{Estimates for derivatives of the curvature}
Similar to the formal subsection, we will present local $L^2$-derivative estimates for the curvature $F_{\nabla_t}$.

From the evolution equation \eqref{ef2}, we have:
\begin{proposition} \label{ec}
Suppose $(\nabla_t,u_t)$ is a solution to the Yang-Mills-Higgs $k$-flow \eqref{YMHKS} defined on $M\times [0,T)$. Then
\begin{equation*}
\begin{split}
\frac{\partial}{\partial t}\|\gamma^{s/2}\nabla_t^{(l)}F_{\nabla_t}\|_{L^2}^2
&=2(-1)^{k}\int_M \Big\langle \Delta_{\nabla_t}^{(k+1)}\nabla_t^{(l)}F_{\nabla_t},\gamma^s\nabla_t^{(l)}F_{\nabla_t}\Big\rangle\\
&\quad +\int_M \Big\langle \sum_{v=0}^{2k+l}(P^{(v)}_1[F_{\nabla_t}]+P_2^{(v)}[F_{\nabla_t}])+P_3^{(2k+l-2)}[F_{\nabla_t}],\gamma^s \nabla_t^{(l)}F_{\nabla_t}\Big\rangle\\
&\quad +\int_M \Big\langle \sum_{i=0}^k\nabla_t^{(l)}D_{\nabla_t}\nabla_t^{*(i)}(\nabla_t^{(k+1)}u_t\ast \nabla_t^{(k-i)}u_t),\gamma^s \nabla_t^{(l)}F_{\nabla_t}\Big\rangle\\
&\quad +\int_M \Big\langle \sum_{j=0}^{l-1}\sum_{i=0}^k\Big[\nabla^{(j)}_t\Big(\nabla^{*(i)}_t(\nabla^{(k+1)}_tu_t\ast \nabla_t^{(k-i)}u_t)\Big)\Big]\ast \nabla_t^{(l-j-1)}F_{\nabla_t},
\gamma^s \nabla_t^{(l)}F_{\nabla_t}\Big\rangle.
\end{split}
\end{equation*}
\end{proposition}

Similar to the proof of Lemma \ref{se1}. We have the following local $L^2$-derivative estimate for the curvature.
\begin{proposition} \label{lec}
Suppose $(\nabla_t,u_t)$ is a solution to the Yang-Mills-Higgs $k$-flow \eqref{YMHKS} defined on $M\times [0,T)$. Assume  $Q=\max\{1,\sup\limits_{t\in [0,T)} |F_{\nabla_t}|\}$, $K=\max\{1,\sup\limits_{t\in [0,T)}|u_t|\}$, $\gamma$ is bump function. Then for $s\geq2(k+l+1)$, there exists $\lambda\in[1,2)$ and $C:=C(\mathrm{dim} (M),\mathrm{rk} (E),s,k,l,g,h,\gamma)$ such that
\begin{equation*}
\begin{split}
\frac{\partial}{\partial t}\|\gamma^{s/2}\nabla_t^{(l)}F_{\nabla_t}\|_{L^2}^2
&\leq -\lambda\|\gamma^{s/2}\nabla_t^{(k+l+1)}F_{\nabla_t}\|^2_{L^2}+CQ^4K^2\|F_{\nabla_t}\|^2_{L^2,\gamma>0}.
\end{split}
\end{equation*}
\end{proposition}

\subsection{Coupled estimates for the curvature and the Higgs field}

As the Yang-Mills-Higgs $k$-flow is a coupled system, we can not obtain a local estimate for the curvature or the Higgs field along. From Proposition \ref{ehf} and Proposition \ref{lec}, we have the following proposition.
\begin{proposition} \label{ce}
Suppose $(\nabla_t,u_t)$ is a solution to the Yang-Mills-Higgs $k$-flow \eqref{YMHKS} defined on $M\times [0,T)$. Assume  $Q=\max\{1,\sup\limits_{t\in [0,T)} |F_{\nabla_t}|\}$, $K=\max\{1,\sup\limits_{t\in [0,T)}|u_t|\}$, $\gamma$ is bump function. Then for $s\geq2(k+l+1)$, there exist $\lambda\in[1,2)$ and $C:=C(\mathrm{dim} (M),\mathrm{rk} (E),s,k,l,g,h,\gamma)$ such that
\begin{equation*}
\begin{split}
\frac{\partial}{\partial t}(\|\gamma^{s/2}\nabla_t^{(l)}F_{\nabla_t}\|_{L^2}^2+\|\gamma^{s/2}\nabla_t^{(l)}u_t\|_{L^2}^2)
&\leq -\lambda(\|\gamma^{s/2}\nabla_t^{(k+l+1)}F_{\nabla_t}\|^2_{L^2}+\|\gamma^{s/2}\nabla_t^{(k+l+1)}u_t\|^2_{L^2})\\
&\quad +CQ^4K^4(\|F_{\nabla_t}\|^2_{L^2,\gamma>0}+\|u_t\|^2_{L^2,\gamma>0}).
\end{split}
\end{equation*}
\end{proposition}

Using the above proposition and following \cite{Ke,Sa}, we can derive estimates of Bernstein-Bando-Shi type.

\begin{proposition} \label{bbs}
Let $q\in \mathbb{N}$ and $\gamma$ be a bump function. Suppose $(\nabla_t,u_t)$ is a solution to the Yang-Mills-Higgs $k$-flow \eqref{YMHKS} defined on $M\times I$. Assume  $Q=\max\{1,\sup\limits_{t\in I} |F_{\nabla_t}|\}$, $K=\max\{1,\sup\limits_{t\in I}|u_t|\}$, and choose $s\geq (k+1)(q+1)$. Then for $t\in [0,T)\subset I$ with $T<\frac{1}{(QK)^{4}}$, there exists a positive constant $C_q:=C_q(\mathrm{dim} (M),\mathrm{rk} (E),q,k,s,g,h,\gamma)\in \mathbb{R}_{>0}$ such that
\begin{equation} \label{bbsi}
\|\gamma^{s}\nabla_t^{(q)}F_{\nabla_t}\|_{L^2}^2+\|\gamma^{s}\nabla_t^{(q)}u_t\|_{L^2}^2
\leq C_qt^{-\frac{q}{k+1}}\sup_{t\in[0,T)}(\|F_{\nabla_t}\|^2_{L^2,\gamma>0}+\|u_t\|^2_{L^2,\gamma>0}).
\end{equation}
\end{proposition}

\begin{proof}
Set $a_q:=1$ and let $\{a_{l}\}_{l=0}^{q-1}\subset \mathbb{R}$ be coefficients to be determined. Define
$$
\Phi(t):=\sum_{l=0}^qa_lt^l(\|\gamma^{s}\nabla_t^{((k+1)l)}F_{\nabla_t}\|_{L^2}^2+\|\gamma^{s}\nabla_t^{((k+1)l)}u_t\|_{L^2}^2).
$$
Differentiating $\Phi$ and applying Proposition \ref{ce}, we have
\begin{equation*}
\begin{split}
\frac{\partial}{\partial t}\Phi(t)
&=\sum_{l=1}^qla_lt^{l-1}(\|\gamma^{s}\nabla_t^{((k+1)l)}F_{\nabla_t}\|_{L^2}^2+\|\gamma^{s}\nabla_t^{((k+1)l)}u_t\|_{L^2}^2)\\
&\quad +\sum_{l=0}^qa_lt^{l}\frac{\partial}{\partial t}(\|\gamma^{s}\nabla_t^{((k+1)l)}F_{\nabla_t}\|_{L^2}^2+\|\gamma^{s}\nabla_t^{((k+1)l)}u_t\|_{L^2}^2)\\
&\quad \leq \sum_{l=0}^{q-1}(l+1)a_{l+1}t^{l}(\|\gamma^{s}\nabla_t^{((k+1)(l+1))}F_{\nabla_t}\|_{L^2}^2+\|\gamma^{s}\nabla_t^{((k+1)(l+1))}u_t\|_{L^2}^2)\\
&\quad +\sum_{l=0}^qa_lt^{l}\Big[-(\|\gamma^{s}\nabla_t^{((k+1)(l+1))}F_{\nabla_t}\|^2_{L^2}+\|\gamma^{s}\nabla_t^{((k+1)(l+1))}u_t\|^2_{L^2})\\
&\quad +CQ^4K^4(\|F_{\nabla_t}\|^2_{L^2,\gamma>0}+\|u_t\|^2_{L^2,\gamma>0})\Big]\\
&=-t^q(\|\gamma^{s}\nabla_t^{((k+1)(q+1))}F_{\nabla_t}\|_{L^2}^2+\|\gamma^{s}\nabla_t^{((k+1)(q+1))}u_t\|_{L^2}^2)\\
&\quad +\sum_{l=0}^{q-1}[a_{l+1}(l+1)-a_l]t^l(\|\gamma^{s}\nabla_t^{((k+1)(l+1))}F_{\nabla_t}\|_{L^2}^2+\|\gamma^{s}\nabla_t^{((k+1)(l+1))}u_t\|_{L^2}^2)\\
&\quad +CQ^4K^4\sum_{l=0}^qa_lt^l(\|F_{\nabla_t}\|^2_{L^2,\gamma>0}+\|u_t\|^2_{L^2,\gamma>0}).
\end{split}
\end{equation*}
For $l=0,\cdots, q-1$, we choose constants satisfy $a_l\geq \frac{q!}{l!}$ , then $a_{l+1}(l+1)-a_l\leq 0$. Noting that $T<\frac{1}{(QK)^{4}}$ and choosing $C_{(k+1)q}\geq C(\sum_{l=0}^qa_l)$, we have
$$\frac{\partial}{\partial t}\Phi(t)\leq C_{(k+1)q}Q^4K^4(\|F_{\nabla_t}\|^2_{L^2,\gamma>0}+\|u_t\|^2_{L^2,\gamma>0}),$$
which means
$$\Phi(t)-\Phi(0)\leq C_{(k+1)q}Q^4K^4\int_0^t(\|F_{\nabla_{\tau}}\|^2_{L^2,\gamma>0}+\|u_{\tau}\|^2_{L^2,\gamma>0})d\tau.$$
Therefore
\begin{equation*}
\begin{split}
&t^q(\|\gamma^{s}\nabla_t^{((k+1)q)}F_{\nabla_t}\|_{L^2}^2+\|\gamma^{s}\nabla_t^{((k+1)q)}u_t\|_{L^2}^2)\\
&\leq C_{(k+1)q}TQ^4K^4\sup_{t\in[0,T)}(\|F_{\nabla_t}\|^2_{L^2,\gamma>0}+\|u_t\|^2_{L^2,\gamma>0})+\Phi(0)\\
&\leq C_{(k+1)q}\sup_{t\in[0,T)}(\|F_{\nabla_t}\|^2_{L^2,\gamma>0}+\|u_t\|^2_{L^2,\gamma>0})+q!(\|F_{\nabla_0}\|^2_{L^2,\gamma>0}+\|u_{0}\|^2_{L^2,\gamma>0}),
\end{split}
\end{equation*}
which means
$$
\|\gamma^{s}\nabla_t^{((k+1)q)}F_{\nabla_t}\|_{L^2}^2+\|\gamma^{s}\nabla_t^{((k+1)q)}u_t\|_{L^2}^2\leq C_{(k+1)q}t^{-q}\sup_{t\in[0,T)}(\|F_{\nabla_t}\|^2_{L^2,\gamma>0}+\|u_t\|^2_{L^2,\gamma>0}).
$$

To complete the proof, it remains to consider $\|\gamma^s\nabla_t^{((k+1)l+w)}F_{\nabla_t}\|^2_{L^2}+\|\gamma^s\nabla_t^{((k+1)l+w)}u_t\|^2_{L^2}$, where $l\in \mathbb{N}\cup \{0\}$ and $w\in [1,k]\cap \mathbb{N}$. From Lemma \ref{ii}, combined with $T<\frac{1}{Q^4K^4}$, we have
\begin{equation*}
\begin{split}
&\|\nabla_t^{((k+1)l+w)}F_{\nabla_t}\|^2_{L^2}+\|\nabla_t^{((k+1)l+w)}u_t\|^2_{L^2}\\
&\leq \varepsilon(\|\nabla_t^{((k+1)(l+1))}F_{\nabla_t}\|^2_{L^2}+\|\nabla_t^{((k+1)(l+1))}u_t\|^2_{L^2})\\
&~~+C_{\varepsilon}(\|F_{\nabla_t}\|^2_{L^2,\gamma>0}+\|u_t\|^2_{L^2,\gamma>0})\\
&\leq C_{(k+1)(l+1)}t^{-l-1}\sup_{t\in[0,T)}(\|F_{\nabla_t}\|^2_{L^2,\gamma>0}+\|u_t\|^2_{L^2,\gamma>0})\\
&~~+C_{\varepsilon}t^{-l-1}(\|F_{\nabla_t}\|^2_{L^2,\gamma>0}+\|u_t\|^2_{L^2,\gamma>0})\\
&\leq Ct^{-\frac{(k+1)(l+1)+w}{k+1}}\sup_{t\in[0,T)}(\|F_{\nabla_t}\|^2_{L^2,\gamma>0}+\|u_t\|^2_{L^2,\gamma>0}).
\end{split}
\end{equation*}
Therefore, we have established \eqref{bbsi} for all $q$.

\end{proof}

The following corollary is a direct consequence of the above inequality, which will be used in the blow-up analysis. The proof relies on embedding $W^{p,2}\subset C^0$ provided $p>\frac{n}{2}$, and then uses the Kato's inequality $|d|u_t||\leq |\nabla_tu_t|$. More details can be found in Kelleher's paper (\protect {\cite[Corollary 3.14]{Ke}}).

\begin{corollary} \label{ce1}
Suppose $(\nabla_t,u_t)$ is a solution to the Yang-Mills-Higgs $k$-flow \eqref{YMHKS} defined on $M\times [0,\tau]$.  Set $\bar{\tau}:=\min\{\tau,1\}$. Assume $Q=\max\{1,\sup\limits_{t\in [0,\bar{\tau}]} |F_{\nabla_t}|\}$, $K=\max\{1,\sup\limits_{t\in [0,\bar{\tau}]}|u_t|\}$. Suppose $\gamma$ is bump function. For $s,l\in \mathbb{N}$ with $s\geq (k+1)(l+1)$ there exists $C_l:=C_l(\dim (M),\mathrm{rk}(E),K,Q,s,k,l,\tau,g,h,\gamma)\in \mathbb{R}_{>0}$ such that
\begin{equation*}
\sup_M\Big(|\gamma^{s}\nabla_{\bar{\tau}}^{(l)}F_{\nabla_{\bar{\tau}}}|^2+|\gamma^{s}\nabla_{\bar{\tau}}^{(l)}u_{\bar{\tau}}|^2\Big)
\leq C_l\sup_{M\times [0,{\bar{\tau}})}\Big(\|F_{\nabla_t}\|^2_{L^2,\gamma>0}+\|u_t\|^2_{L^2,\gamma>0}\Big).
\end{equation*}
\end{corollary}

\begin{remark}
Corollary \ref{ce1} has no dependency on the initial data $(\nabla_0,u_0)$.
\end{remark}

Using Corollary \ref{ce1}, we have the the following corollary, which can be used for finding obstructions to long time existence.

\begin{corollary} \label{ce2}
Suppose $(\nabla_t,u_t)$ is a solution to the Yang-Mills-Higgs $k$-flow \eqref{YMHKS} defined on $M\times [0,T)$ for $T\in[0,+\infty)$.  Assume
 $$Q=\max\{1,\sup\limits_{t\in [0,T)} |F_{\nabla_t}|,\sup\limits_{t\in [0,T)} \|F_{\nabla_t}\|_{L^2}\}$$
 and
 $$ K=\max\{1,\sup\limits_{t\in [0,T)}|u_t|,\sup\limits_{t\in [0,T)}\|u_t\|_{L^2}\}$$
 are finite. Suppose $\gamma$ is bump function. Then for $t\in [0,T)$, $s,l\in \mathbb{N}$ with $s\geq (k+1)(l+1)$, there exists $C_l:=C_l(\nabla_0,u_0,\dim (M),\mathrm{rk}(E),K,Q,s,k,l,g,h,\gamma)\in \mathbb{R}_{>0}$ such that
\begin{equation*}
\sup_{M\times[0,T)}\Big(|\gamma^{s}\nabla_t^{(l)}F_{\nabla_{t}}|^2+|\gamma^{s}\nabla_t^{(l)}u_t|^2\Big)\leq C_l.
\end{equation*}
\end{corollary}

\subsection{Long time existence obstruction}

In this section, we will use Corollary \ref{ce2} to show that the only obstruction to long time existence of the Yang-Mills-Higgs $k$-flow \eqref{YMHKS} is a lack of supremal bound on $|F_{\nabla_t}|+|u_t|$.

We first recall Kelleher's lemma:

\begin{lemma} (\protect {\cite[Lemma 3.17]{Ke}}) \label{KE2}
Let $\nabla,\widetilde{\nabla} \in \mathcal{A}_E$ and set $\Upsilon:=\widetilde{\nabla}-\nabla$. Then for all $\xi$ in some tensor product of $TM,E$, and their corresponding duals,
$$
\widetilde{\nabla}^{(l)}\xi=\nabla^{(l)}\xi+\sum_{j=0}^{l-1}\sum_{i=0}^j(\widetilde{P}^{(i)}_{l-1-i}[\Upsilon]\ast \widetilde{P}^{(j-i)}_{1}[\xi]).
$$
\end{lemma}

For later use, given a 1-parameter family $(\nabla_t,u_t)$ over $M\times [0,T)$ with $T<+\infty$, set
$$
\Upsilon_s:=\int_0^s\frac{\partial \nabla_t}{\partial t}dt,\ \ \Psi_s:=\int_0^s\frac{\partial u_t}{\partial t}dt.
$$

\begin{proposition} \label{obst1}
Suppose $(\nabla_t,u_t)$ is a solution to the Yang-Mills-Higgs $k$-flow \eqref{YMHKS} defined on $M\times [0,T)$ for $T\in[0,+\infty)$. Suppose that for all $l\in \mathbb{N}\cup \{0\}$ there exists $C_l\in \mathbb{R}_{>0}$ such that
$$\max\Big\{\sup_{M\times[0,T)}|\nabla_t^{(l)}[\frac{\partial \nabla_t}{\partial t}]|, \sup_{M\times[0,T)}|\nabla_t^{(l)}[\frac{\partial u_t}{\partial t}]|\Big\}\leq C_l.$$
Then $\lim_{t\rightarrow T}(\nabla_t,u_t)=(\nabla_T,u_T)$ exists and is smooth.
\end{proposition}

\begin{proof}
For all $s\leq T$,
$$|\Upsilon_s|=|\int_0^s\frac{\partial \nabla_t}{\partial t}dt|\leq TC_0, \ \ |\Psi_s|=|\int_0^s\frac{\partial u_t}{\partial t}dt|\leq TC_0,$$
which means $(\nabla_T,u_T)$ is continuous.

Next, we demonstrate that $(\nabla_T,u_T)$ is smooth. The proof proceeds by induction on $l$ satisfying $|\nabla_0^{(l)}[\Psi_T]|+|\nabla_0^{(l)}[\Upsilon_T]|<+\infty$. For the base case,
\begin{equation*}
\begin{split}
|\nabla_0[\Psi_s]|
& =\Big|\int_0^s\nabla_0[\frac{\partial u_t}{\partial t}]dt\Big|\leq \int_0^s\Big(|\nabla_t[\frac{\partial u_t}{\partial t}]|+C|\Upsilon_t||\frac{\partial u_t}{\partial t}|\Big)dt\\
&\leq TC_1+CT^2C_0^2<+\infty.
\end{split}
\end{equation*}
We also have
$$|\nabla_0[\Upsilon_s]|<+\infty.$$

Now suppose the induction hypothesis is satisfied for $\{1,\cdots,l-1\}$. Expanding $\nabla_0^{(l)}[\Psi_s]$, applying Lemma \ref{KE2} and then by assumption,
\begin{equation*}
|\nabla^{(l)}_0[\Psi_s]|=\int_0^s\Big(|\nabla_t^{(l)}[\frac{\partial u_t}{\partial t}]|+\sum_{j=0}^{l-1}\sum_{i=0}^j(P^{(i)}_{l-i-1}[\Upsilon_t]\ast P_1^{(j-i)}[\frac{\partial u_t}{\partial t}])\Big)dt< +\infty,
\end{equation*}
where the notation $P$ are taken with respect to $\nabla_t$. Similarly,
$$|\nabla^{(l)}_0[\Upsilon_s]|<+\infty.$$
Since the bounds are uniform for all $t\in[0,T)$ and $\Upsilon_s, \Phi_s$ are continuous, we have that
$$|\nabla_0^{(l)}[\Psi_T]|+|\nabla_0^{(l)}[\Upsilon_T]|<+\infty.$$
Thus $\Upsilon_T, \Phi_T$ are smooth. This completes the proof.
\end{proof}

Using Proposition \ref{obst1}, we are ready to prove the main result in this subsection.

\begin{theorem} \label{obst2}
Suppose $(\nabla_t,u_t)$ is a solution to the Yang-Mills-Higgs $k$-flow \eqref{YMHKS} for some maximal $T<+\infty$. Then
$$\sup_{M\times[0,T)}(|F_{\nabla_t}|+|u_t|)=+\infty.$$
\end{theorem}

\begin{proof}
Suppose to the contrary that
$$\sup_{M\times[0,T)}(|F_{\nabla_t}|+|u_t|)<+\infty.$$
By Corollary \ref{ce2}, for all $t\in [0,T)$  and $l\in \mathbb{N}\cup\{0\}$, we have $\sup_{M}\Big(|\nabla_t^{(l)}F_{\nabla_{t}}|^2+|\nabla_t^{(l)}u_t|^2\Big)$ is uniformly bounded and so by Proposition \ref{obst1}, $\lim_{t\rightarrow T}(\nabla_t,u_t)=(\nabla_T,u_T)$ exists and is smooth. However, by local existence (Theorem \ref{le}), there exists $\epsilon>0$ such that $(\nabla_t,u_t)$ exists over the extended domain $[0,T+\epsilon)$, which contradicts the assumption that $T$ was maximal.

\end{proof}

\vskip 1cm

\section{Blow-up analysis} \label{ba}

In this section, we will address the possibility of Yang-Mills-Higgs $k$-flow singularities given no bound on $|F_{\nabla_t}|+|u_t|$. To begin with, we will establish some preliminary scaling laws for  Yang-Mills-Higgs $k$-flow.

\begin{proposition} \label{rf}
Suppose $(\nabla_t,u_t)$ is a solution to the Yang-Mills-Higgs $k$-flow \eqref{YMHKS} defined on $M\times [0,T)$. We define the 1-parameter family $\nabla_t^{\rho}$ with local coefficient matrices given by
$$\Gamma_t^{\rho}(x):=\rho \Gamma_{\rho^{2(k+1)}t}(\rho x),$$
where $\Gamma_t(x)$ are local coefficiecnt matrices of $\nabla_t$. We define the $\rho$-scaled Higgs field $u_t^{\rho}$ by
$$u_t^{\rho}(x):=\rho u_{\rho^{2(k+1)}t}(\rho x).$$
Then $(\nabla_t^{\rho},u_t^{\rho})$ is also a solution to the Yang-Mills-Higgs $k$-flow \eqref{YMHKS} defined on $[0,\frac{1}{\rho^{2(k+1)}}T)$.
\end{proposition}

\begin{proof}
We start by computing time derivatives of the scaled connection and Higgs field
\begin{equation*}
\begin{split}
\frac{\partial \nabla^{\rho}}{\partial t}(x,t)&=\rho^{2k+3}\frac{\partial \nabla}{\partial t}(\rho x,\rho^{2(k+1)}t),\\
\frac{\partial u^{\rho}}{\partial t}(x,t)&=\rho^{2k+3}\frac{\partial u}{\partial t}(\rho x,\rho^{2(k+1)}t).
\end{split}
\end{equation*}
Thus the desired scaling law holds through the Yang-Mills-Higgs $k$-flow.

\end{proof}

Next we will show that in the case that the curvature coupled with Higgs field is blowing up, as one approaches the maximal time, one can extract a blow-up limit. The proof will closely follow the arguments in (\protect {\cite[Proposition 3.25]{Ke}}).

\begin{theorem} \label{bll}
Suppose $(\nabla_t,u_t)$ is a solution to the Yang-Mills-Higgs $k$-flow \eqref{YMHKS} defined on some maximal time interval $[0,T)$ with $T< +\infty$. Then there exists a blow-up sequence $(\nabla^i_t,u^i_t)$ and converges pointwise to a smooth solution $(\nabla^{\infty}_t,u^{\infty}_t)$ to the Yang-Mills-Higgs $k$-flow \eqref{YMHKS} defined on the domain $\mathbb{R}^n\times \mathbb{R}_{<0}$.
\end{theorem}

\begin{proof}
From Theorem \ref{obst2}, we must have
$$\lim_{t\rightarrow T}\sup_M\Big(|F_{\nabla_t}|+\langle u_t,u_t\rangle\Big)=+\infty.$$
Therefore, we can choose a sequence of times $t_i\nearrow T$ within $[0,T)$, and a sequence of points $x_i$, such that
$$|F_{\nabla_{t_i}}(x_i)|+\langle u_{t_i}(x_i),u_{t_i}(x_i)\rangle =\sup_{M\times [0,t_i]}\Big(|F_{\nabla_t}|+\langle u_t,u_t\rangle \Big).$$
Let $\{\rho_i\}\subset \mathbb{R}_{>0}$ be constants to be determined. Define $\nabla_t^i(x)$ by
$$\Gamma_t^i(x)=\rho_i^{\frac{1}{2(k+1)}}\Gamma_{\rho_i t+t_i}(\rho_i^{\frac{1}{2(k+1)}} x+x_i)$$
and
$$u_t^i(x)=\rho_i^{\frac{1}{2(k+1)}}u_{\rho_i t+t_i}(\rho_i^{\frac{1}{2(k+1)}} x+x_i).$$
By Proposition \ref{rf}, $(\nabla^i_t,u^i_t)$ are also solutions to Yang-Mills-Higgs $k$-flow \eqref{YMHKS} and the domain for each $(\nabla^i_t,u^i_t)$ is $B_o(\rho_i^{-\frac{1}{2(k+1)}})\times [-\frac{t_i}{\rho_i},\frac{T-t_i}{\rho_i})$. We observe that
$$F_t^i(x):=F_{\nabla_t^i}(x)=\rho_i^{\frac{1}{k+1}}F_{\nabla_{\rho_i t+t_i}}(\rho_i^{\frac{1}{2(k+1)}}x+x_i),$$
which means
\begin{equation*}
\begin{split}
&\sup_{t\in[-\frac{t_i}{\rho_i},\frac{T-t_i}{\rho_i})}\Big(|F_t^i(x)|+| u_t^i(x)|^2\Big)\\
&=\rho_i^{\frac{1}{k+1}}\sup_{t\in[-\frac{t_i}{\rho_i},\frac{T-t_i}{\rho_i})}\Big(|F_{\nabla_{\rho_i t+t_i}}(\rho_i^{\frac{1}{2(k+1)}} x+x_i)|+| u_{\rho_i t+t_i}(\rho_i^{\frac{1}{2(k+1)}} x+x_i)|^2\Big)\\
&=\rho_i^{\frac{1}{k+1}}\sup_{t\in[0,t_i]}\Big(|F_{\nabla_t}(x)|+| u_t(x)|^2\Big)\\
&=\rho_i^{\frac{1}{k+1}}\Big(|F_{\nabla_{t_i}}(x_i)|+|u_{t_i}(x_i)|^2\Big).
\end{split}
\end{equation*}
Therefore, setting
$$\rho_i=\Big(|F_{\nabla_{t_i}}(x_i)|+| u_{t_i}(x_i)|^2\Big)^{-(k+1)},$$
which gives
\begin{equation} \label{model}
1=|F^i_0(0)|+|u_0^i(0)|^2=\sup_{t\in[-\frac{t_i}{\rho_i},0]}\Big(|F_t^i(x)|+| u_t^i(x)|^2\Big).
\end{equation}

Now, we are ready to construct smoothing estimates for the sequence $(\nabla^i_t,u^i_t)$. Let $y\in \mathbb{R}^n$, $\tau\in \mathbb{R}_{\leq 0}$. For any $s\in \mathbb{N}$,
$$\sup_{t\in [\tau-1,\tau]}\Big(|\gamma^s_y F^i_t(x)|+|\gamma^s_y u^i_t(x)|^2\Big)\leq 1.$$
By Corollary \ref{ce1}, for all $q\in \mathbb{N}$, one may choose $s\geq(k+1)(q+1)$ so that there exists positive constant $C_q$ such that
\begin{equation*}
\begin{split}
&\sup_{x\in B_y(\frac{1}{2})}\Big(|(\nabla^i_{\tau})^{(q)} F^i_{\tau}(x)|+|(\nabla^i_{\tau})^{(q)}  u^i_{\tau}(x)|\Big)\\
&\leq\sup_{x\in B_y(1)}\Big(|\gamma^s_y(\nabla^i_{\tau})^{(q)} F^i_{\tau}(x)|+|\gamma^s_y(\nabla^i_{\tau})^{(q)}  u^i_{\tau}(x)|\Big)\\
&\leq C_q.
 \end{split}
\end{equation*}
Then by the Coulomb Gauge Theorem of Uhlenbeck \cite[Theorem 1.3]{U} (also see \cite{HT1}) and Gauge Patching Theorem \cite[Corollary 4.4.8]{DK}, passing to a subsequence (without changing notation) and in an appropriate gauge, $(\nabla^i_t,u^i_t)\rightarrow (\nabla^{\infty}_t,u^{\infty}_t)$ in $C^{\infty}$.

\end{proof}

\vskip 1cm

\section{Energy estimates} \label{ee}

In this section, we will prove that both the Yang-Mills-Higgs $k$-energy and the Yang-Mills-Higgs energy are bounded along the Yang-Mills-Higgs $k$-flow.

We first show that the Yang-Mills-Higgs $k$-energy is bounded.

\begin{proposition} \label{hoymhe}
Suppose $(\nabla_t,u_t)$ is a solution to the Yang-Mills-Higgs $k$-flow \eqref{YMHKS} defined on $M\times [0,T)$. The Yang-Mills-Higgs $k$-energy \eqref{YMHKF} is decreasing along the flow \eqref{YMHKS}.
\end{proposition}

\begin{proof}
Direct calculation yields
$$\frac{\partial}{\partial t} \mathcal{YMH}_k(\nabla_t,u_t)=-\Big(\|\frac{\partial \nabla_t}{\partial t}\|^2_{L^2}+\|\frac{\partial u_t}{\partial t}\|^2_{L^2}\Big)\leq 0.$$
\end{proof}

For later use, we first prove an $L^2$-bound for the Higgs field $u_t$.

\begin{lemma}
Suppose $(\nabla_t,u_t)$ is a solution to the Yang-Mills-Higgs $k$-flow \eqref{YMHKS} defined on $M\times [0,T)$. We have
$$
\sup_{t\in[0,T)}\|u_t\|_{L^2}<+\infty.
$$
\end{lemma}

\begin{proof}
Direct calculation yields
\begin{equation*}
\begin{split}
\frac{\partial}{\partial t}\int_M \langle u_t,u_t\rangle
&=2\int_M \langle u_t, -\nabla_t^{*(k+1)}\nabla_t^{(k+1)}u_t\rangle\\
&=-2\int_M|\nabla_t^{(k+1)}u_t|^2\leq 0.
\end{split}
\end{equation*}
\end{proof}

Using the above lemma, we can show that the Yang-Mills-Higgs energy is bounded along the Yang-Mills-Higgs $k$-flow.

\begin{proposition} \label{ymhe}
Suppose $(\nabla_t,u_t)$ is a solution to the Yang-Mills-Higgs $k$-flow \eqref{YMHKS} defined on $M^4\times [0,T)$ with $T< +\infty$, then the Yang-Mills-Higgs energy
$$
\mathcal{YMH}(\nabla_t,u_t)=\frac{1}{2}\int_M\Big[|F_{\nabla_t}|^2+|\nabla_t u_t|^2\Big]d\mathrm{vol}_g
$$
is bounded along the flow \eqref{YMHKS}.
\end{proposition}

\begin{proof}

Direct calculation yields
\begin{equation*}
\begin{split}
\frac{\partial}{\partial t}\mathcal{YMH}(\nabla_t,u_t)&=\int_M \langle D^*_{\nabla_t}F_{\nabla_t}+\nabla_t u_t\otimes u^*_t, \frac{\partial \nabla_t}{\partial t}\rangle+\int_M \langle \frac{\partial u_t}{\partial t},
\nabla_t^*\nabla_t u_t\rangle \\
&\leq \int_M\Big[ |\frac{\partial \nabla_t}{\partial t}|^2+|\frac{\partial u_t}{\partial t}|^2+C(|\nabla_t F_{\nabla_t}|^2+|\nabla_t^{(2)}u_t|^2+|\nabla_t u_t|^2|u_t|^2)\Big]\\
&\leq -\frac{\partial }{\partial t}\mathcal{YMH}_k(\nabla_t,u_t)+C(\|\nabla_t^{(k)}F_{\nabla_t}\|^2_{L^2}+\|\nabla_t^{(k+1)}u_t\|^2_{L^2})\\
&\quad +\varepsilon(\|F_{\nabla_t}\|^2_{L^2}+\|\nabla_tu_t\|^2_{L^2})+C(\|\nabla^{(k+1)}_t u_t\|^4_{L^2}+\|u_t\|^4_{L^2}),
\end{split}
\end{equation*}
where the last inequality used Lemma \ref{ii}, the H\"{o}lder inequality and the following Sobolev inequalities
\begin{equation*} \label{s1}
\|u_t\|^2_{L^4}\leq C (\|\nabla_t u_t\|^2_{L^2}+\| u_t\|^2_{L^2}),
\end{equation*}
\begin{equation*} \label{s2}
\|\nabla_t u_t\|^2_{L^4}\leq C (\|\nabla^{(2)}_t u_t\|^2_{L^2}+\| \nabla_tu_t\|^2_{L^2}),
\end{equation*}
here $C$ is a constant independent of $t\in [0,T)$. Therefore, we have
\begin{equation} \label{s3}
\begin{split}
\mathcal{YMH}(\nabla_t,u_t)-\mathcal{YMH}(\nabla_0,u_0)
&\leq CT \Big(\mathcal{YMH}_k(\nabla_0,u_0)+\mathcal{YMH}^2_k(\nabla_0,u_0)+\|u_t\|_{L^2}^4\Big)\\
&~~+\varepsilon T\sup_{t\in [0,T)}\mathcal{YMH}(\nabla_t,u_t).
\end{split}
\end{equation}

Next, we will borrow an argument in Saratchandran's paper \cite[Theorem 5.3]{Sa}. Suppose that there exists $t_m\rightarrow T$ such that
$$\lim_{m\rightarrow +\infty}\mathcal{YMH}(\nabla_{t_m},u_{t_m})\rightarrow +\infty.$$
 By throwing out some of $t_m$, we can assume that $\mathcal{YMH}(\nabla_{t_m},u_{t_m})>\mathcal{YMH}(\nabla_{t_{m'}},u_{t_{m'}})$ for $m\geq m'$, and that $t_m\geq t_{m'}$ when $m\geq m'$. Partition $[0,T)=[t_0,t_1]\cup [t_1,t_2]\cup\cdots\cup[t_k,t_{k+1}]\cup\cdots$ with $t_0=0$. Define $s_i\in [t_i,t_{i+1}]$ by $\sup_{t\in[t_i,t_{i+1}]}\mathcal{YMH}(\nabla_{t},u_{t})=\mathcal{YMH}(\nabla_{s_i},u_{s_i})$. It is easy to see that $s_i\rightarrow T$, and $\mathcal{YMH}(\nabla_{s_i},u_{s_i})\rightarrow +\infty$ as $i\rightarrow +\infty$. Furthermore, $\mathcal{YMH}(\nabla_{s_j},u_{s_j})\leq \mathcal{YMH}(\nabla_{s_i},u_{s_i})$ when $j\leq i$. Then substitute $s_i$ for $t$ in \eqref{s3}, we have
\begin{equation*}
\begin{split}
&\mathcal{YMH}(\nabla_{s_i},u_{s_i})-\mathcal{YMH}(\nabla_0,u_0)-\varepsilon T\mathcal{YMH}(\nabla_{s_i},u_{s_i})\\
&\leq C T\Big(\mathcal{YMH}_k(\nabla_0,u_0)+\mathcal{YMH}^2_k(\nabla_0,u_0)+\|u_t\|_{L^2}^4\Big),
\end{split}
\end{equation*}
which means
$$\mathcal{YMH}(\nabla_{s_i},u_{s_i})\leq \frac{1}{1-\varepsilon T}CT \Big(\mathcal{YMH}_k(\nabla_0,u_0)+\mathcal{YMH}^2_k(\nabla_0,u_0)+\|u_t\|_{L^2}^4+\mathcal{YMH}(\nabla_0,u_0)\Big).$$
The right hand side of the above inequality is finite, and it is independent of $i$. After taking $i\rightarrow +\infty$ on the left, we conclude a contradiction. Thus no such $\{t_m\}$ exists and the result follows.

\end{proof}

\vskip 1cm

\section{Proof of Theorem \ref{mth}} \label{pot}

In this section, we will complete the proof of Theorem \ref{mth}. To accomplish this, we first show that the $L^p$-norm controls the $L^{\infty}$-norm by blow-up analysis.

\begin{proposition} \label{lptli}
Suppose $(\nabla_t,u_t)$ is a solution to the Yang-Mills-Higgs $k$-flow \eqref{YMHKS} defined on $M^n\times [0,T)$ and
$$\sup_{t\in[0,T)}(\|F_{\nabla_t}\|_{L^p}+\|\langle u_t,u_t\rangle\|_{L^p})<+\infty.$$
If $\dim(M)<2p$, then
$$\sup_{t\in[0,T)}(\|F_{\nabla_t}\|_{L^{\infty}}+\|\langle u_t,u_t\rangle\|_{L^{\infty}})<+\infty.$$
\end{proposition}

\begin{proof}
So as to the obtain a contradiction, assume
 $$\sup_{t\in[0,T)}(\|F_{\nabla_t}\|_{L^{\infty}}+\|\langle u_t,u_t\rangle\|_{L^{\infty}})=+\infty.$$
As we did in Theorem \ref{bll}, we can construct a blow-up sequence $(\nabla^i_t,u^i_t)$, with blow-up limit $(\nabla^{\infty}_t,u^{\infty}_t)$. Noting that \eqref{model}, by Fatou's lemma and natural scaling law,
\begin{equation*}
\begin{split}
\|F_{\nabla^{\infty}_t}\|^p_{L^p}+\|\langle u^{\infty}_t, u^{\infty}_t\rangle \|^p_{L^p}
&\leq \lim_{i\rightarrow +\infty}\inf(\|F_{\nabla^{i}_t}\|^p_{L^p}+\|\langle u^{i}_t, u^i_t\rangle \|^p_{L^p})\\
&\leq \lim_{i\rightarrow +\infty}\rho_i^{\frac{2p-n}{2k+2}}(\|F_{\nabla_t}\|^p_{L^p}+\|\langle u_t,u_t\rangle\|^p_{L^p}).
\end{split}
\end{equation*}
Since $\lim_{i\rightarrow +\infty}\rho_i^{\frac{2p-n}{2k+2}}=0$ when $2p>n$, the right hand side of the above inequality tends to zero, which is a contradiction since the blow-up limit has non-vanishing curvature.
\end{proof}

Now we are ready to give the proof of Theorem \ref{mth}.

{\bf Proof of Theorem \ref{mth}}. Since $\dim(M)=4$, in order to use Proposition \ref{lptli}, we need $p>2$. By the Sobolev embedding theorem, we solve for $p$ such that $W^{k,2}\subset L^{2p}$, then $k>1$. In this case, using Lemma \ref{ii} we have
\begin{equation*}
\begin{split}
&\|F_{\nabla_t}\|_{L^{p}}+\|\langle u_t,u_t\rangle\|_{L^{2p}}\\
&\quad \leq CS_{k,p}\sum_{j=0}^k(\|\nabla_t^{(j)}F_{\nabla_t}\|^2_{L^2}+\|\nabla_t^{(j)}u_t\|^2_{L^2}+1)\\
&\quad \leq CS_{k,p}(\|\nabla_t^{(k)}F_{\nabla_t}\|^2_{L^2}+\|F_{\nabla_t}\|^2_{L^2}+\|\nabla_t^{(k+1)}u_t\|^2_{L^2}+\|u_t\|^2_{L^2}+1)\\
&\quad \leq CS_{k,p}(\mathcal{YMH}_k(\nabla_t,u_t)+\mathcal{YMH}(\nabla_t,u_t)+\|u_t\|^2_{L^2}+1).
\end{split}
\end{equation*}
Noting that both $\mathcal{YMH}_k(\nabla_t,u_t)$ and $\mathcal{YMH}(\nabla_t,u_t)$ are bounded along the Yang-Mills-Higgs $k$-flow \eqref{YMHKS} (see Proposition \ref{hoymhe} and Proposition \ref{ymhe}), we conclude that the flow exists smoothly for long time.

\begin{remark}
This proof does not conclude that the flow exists at $t=+\infty$, so the Yang-Mills-Higgs $k$-flow may admit singularities at infinite time.
\end{remark}

\begin{remark}
Since $W^{k+1,2}\subset C^0$ when $k>1$, the $C^0$-bound of $u_t$ can be controlled by $\mathcal{YMH}_k(\nabla_t,u_t)$ and $\|u_t\|^2_{L^2}$ via Lemma \ref{ii}. Then the energy estimate and the blow-up become much more easy. We thought the condition $k>1$ may not be optimal, since it might contract the interval of $k$ before solving $W^{k,2}\subset L^{p}$, which is not happened here. We also address this issue in Section \ref{hsi}.
\end{remark}

\vskip 1cm

\section{Concentration phenomena for Yang-Mills-Higgs 1-flow} \label{cp}

In this section, we will show that the long time existence of Yang-Mills-Higgs 1-flow in dimension $4$ is obstructed by the possibility of concentration of the curvature in smaller and smaller balls.

\begin{proposition} \label{cf}
Suppose $(\nabla_t,u_t)$ is a solution to the Yang-Mills-Higgs $1$-flow \eqref{YMHKS} defined on $M^4\times [0,T)$ with $T$ maximal. Then exists some $\epsilon>0$ such that if $\{(x_i,t_i)\}\subset M\times [0,T)$ with $(x_i,t_i)\rightarrow (X,T)$ has the property that
$$\lim_{i\rightarrow +\infty}\Big(|F_{\nabla_{t_i}}(x_i)|+|u_{t_i}(x_i)|\Big)=+\infty,$$
then for all $r>0$,
$$\lim_{i\rightarrow +\infty}\sup\Big(\|F_{\nabla_{t_i}}\|^2_{L^2(B_X(r))}+\|\langle u_{t_i},u_{t_i}\rangle\|^2_{L^2(B_X(r))}\Big)\geq \epsilon,$$
where $B_X(r)$ denotes the geodesic ball of radius $r$ centered at $X$.
\end{proposition}

\begin{proof}
Choose a corresponding blow-up sequence $(\nabla_t^i,u_t^i)$ as described in Proposition \ref{bll} with limit $(\nabla_t^{\infty},u_t^{\infty})$. Then by \eqref{model},
 $$|F^{\infty}_0(0)|+|u_0^{\infty}(0)|^2=1.$$
By the smoothness of $(\nabla_t^{\infty},u_t^{\infty})$, one has that for $(y,t)\in B_0(\delta)\times (-\delta,0]$ we have
$$|F^{\infty}_t(y)|+|u_t^{\infty}(y)|^2\geq \frac{1}{2}.$$
Therefore,
$$\lim_{t\nearrow 0}\sup \Big(\|F^{\infty}_{t}\|^2_{L^2(B_0(\delta))}+\|\langle u^{\infty}_{t},u^{\infty}_{t}\rangle\|^2_{L^2(B_0(\delta))}\Big)\geq \frac{1}{8}\mathrm{Vol}[B_0(\delta)].$$
Conversely, using the computations in Theorem \ref{lptli},
\begin{equation*}
\begin{split}
&\|F^{\infty}_{t}\|^2_{L^2(B_0(\delta))}+\|\langle u^{\infty}_{t},u^{\infty}_{t}\rangle\|^2_{L^2(B_0(\delta))}\\
&\quad =\int_{B_0(\delta)} \lim_{i\rightarrow +\infty} (|F_{\nabla_t^i}|^2+|\langle u_t^i,u_t^i\rangle |^2)d\mathrm{vol}_g\\
&\quad =\lim_{i\rightarrow +\infty}\rho_i^{\frac{2\times 2-4}{2\times 1+2}}\Big(\|F_{\nabla_{t}}\|^2_{L^2(B_{x_i}(\delta\rho_i^{\frac{1}{4}}))}+\|\langle u_{t},u_{t}\rangle\|^2_{L^2(B_{x_i}(\delta\rho_i^{\frac{1}{4}}))}\Big)\\
&\quad =\lim_{i\rightarrow +\infty}\Big(\|F_{\nabla_{t}}\|^2_{L^2(B_{x_i}(\delta\rho_i^{\frac{1}{4}}))}+\|\langle u_{t},u_{t}\rangle\|^2_{L^2(B_{x_i}(\delta\rho_i^{\frac{1}{4}}))}\Big)
\end{split}
\end{equation*}
Since $\lim_{i\rightarrow +\infty}\rho_i^{\frac{1}{4}}=0$, then for any $r>0$ and $i$ large enough so that $\max\{|T-t_i|\}<\delta$,
$$\frac{1}{8}\mathrm{Vol}[B_0(\delta)]\leq \lim_{i\rightarrow +\infty} \sup\Big(\|F_{\nabla_{t_i}}\|^2_{L^2(B_X(r))}+\|\langle u_{t_i},u_{t_i}\rangle\|^2_{L^2(B_X(r))}\Big).$$
Taking $\epsilon=\frac{1}{8}\mathrm{Vol}[B_0(\delta)]$ yields the result.
\end{proof}

Note that the lower bound given by $\epsilon$ is independent of the point about which the blow-up procedure occurred. From Proposition \ref{cf} we have the following theorem.

\begin{theorem}
Let $(E,h)$ be a vector bundle over a closed Riemannian $4$-manifold $(M,g)$. Then there exists a unique smooth solution $(\nabla_t,u_t)$ to the Yang-Mills-Higgs $1$-flow \eqref{YMHKS}, with smooth initial value $(\nabla_0,u_0)$, existing on $[0,T)$ for some maximal $T\in\mathbb{R}_{>0}\cup \{+\infty\}$. If $T<+\infty$, then there exists a sequence $\{(x_i,t_i)\}\subset M\times [0,T)$ with $(x_i,t_i)\rightarrow (X,T)$ and for all $r>0$,
$$\lim_{i\rightarrow +\infty}\sup\Big(\|F_{\nabla_{t_i}}\|^2_{L^2(B_X(r))}+\|\langle u_{t_i},u_{t_i}\rangle\|^2_{L^2(B_X(r))}\Big)\geq \epsilon.$$

\end{theorem}

\vskip 1cm

\section{Higher order Yang-Mills-Higgs functional with Higgs self-interaction} \label{hsi}

In \protect{\cite{JT}}, Jaffe and Taubes studied the following Yang-Mills-Higgs functional:
\begin{equation} \label{JTF}
\mathcal{YMH}(\nabla,u)=\frac{1}{2}\int_M\Big[| F_{\nabla}|^2+|\nabla u|^2+\frac{\lambda}{4}(|u|^2-1)^2\Big]d\mathrm{vol}_g,
\end{equation}
where the constant $\lambda\geq 0$. The term $\frac{\lambda}{8}(|u|^2-1)^2$ is the Higgs self-interaction.

Following the formal sections, we will consider the following Yang-Mills-Higgs $k$-functional with Higgs self-interaction:
\begin{equation} \label{JTF1}
\mathcal{YMH}_k(\nabla,u)=\frac{1}{2}\int_M\Big[|\nabla^{(k)}F_{\nabla}|^2+|\nabla^{(k+1)}u|^2+\frac{\lambda}{4}(|u|^2-1)^2\Big]d\mathrm{vol}_g.
\end{equation}

The associated negative gradient flow of \eqref{JTF1} is the following system
\begin{equation} \label{JTF2}
\begin{cases}
\frac{\partial \nabla_t}{\partial t}&=(-1)^{(k+1)}D^*_{\nabla_t}\Delta_{\nabla_t}^{(k)}F_{\nabla_t}+\sum\limits_{v=0}^{2k-1}P_1^{(v)}[F_{\nabla_t}]\\
&~~+P_2^{(2k-1)}[F_{\nabla_t}]+\sum\limits_{i=0}^k\nabla_t^{*(i)}(\nabla_t^{(k+1)}u_t\ast \nabla_t^{(k-i)}u_t),\\
\frac{\partial u_t}{\partial t}&=-\nabla^{*(k+1)}_t\nabla_t^{(k+1)}u_t+\frac{\lambda}{2}(1-|u_t|^2)u_t.
\end{cases}
\end{equation}

Now we can follow the line of the study of the flow \eqref{YMHKS}.

(1) First of all, the local existence and smoothing estimates can be achieved in a similar way, and we have a same obstruction (Theorem \ref{obst2}) for the long time existence.

(2) It is easy to check the \eqref{JTF1} is decreasing along the flow \eqref{JTF2}.

(3) One can check that $\|u_t\|_{L^2}$ is bounded along the flow \eqref{JTF2}.

(4) After setting $\dim(M)<2(k+1)$, we have the Sobolev embedding $W^{k+1,2}\subset C^0$. Noting that we have $\|\nabla^{(k+1)}u\|^2_{L^2}$ in \eqref{JTF1}. So the $C^0$-bound of $u_t$ can be controlled by  \eqref{JTF1} by using Lemma \ref{ii}.

(5) Once we have $C^0$-bound of $u_t$, it is easy to prove \eqref{JTF} is bounded along the flow \eqref{JTF2}.

(6) Since $u_t$ is bounded along the flow \eqref{JTF2} when $\dim(M)<2(k+1)$, the obstruction to the long-time existence only depends on $F_{\nabla_t}$. Choose a corresponding blow-up sequence $(\nabla_t^i,u_t^i)$ as described in Theorem \ref{bll}. Thus $(\nabla_t^i,u_t^i)$ converges to $(\nabla_t^{\infty},0)$ smoothly. $u_t^{\infty}=0$ since $u_t$ is bounded.

(7) Using the blow-up analysis, we can control the $C^0$-bound of $F_{\nabla_t}$ by its $L^p$-bound with $\dim(M)<2p$.

(8) At last, when $\dim(M)<2(k+2)$, the $C^0$-bound of $F_{\nabla_t}$ can be derived by Lemma \ref{ii}, the blow-up analysis and the Sobolev embedding $W^{k,2}\subset L^p$ with $\dim(M)<2p$. Therefore, we conclude the following theorem:

\begin{theorem} \label{mth2}
Let $(E,h)$ be a vector bundle over a closed Riemannian manifold $(M,g)$. Assume the integer $k$ satisfies $\dim(M)<2(k+1)$, then there exists a unique smooth solution $(\nabla_t,u_t)$ to the Yang-Mills-Higgs $k$-flow \eqref{JTF2} in $M\times [0,+\infty)$ with smooth initial value $(\nabla_0,u_0)$.
\end{theorem}

\begin{remark}
This proof does not conclude that the flow exists at $t=+\infty$, so the Yang-Mills-Higgs $k$-flow \eqref{JTF2} may admit singularities at infinite time. The critical dimension $2(k+1)$ may not be optimal when the Higgs self-interaction is not vanishing. In this case, we have $\|u_t\|_{L^4}^4$ in $\mathcal{YMH}_k(\nabla_t,u_t)$, which is better than $\|u_t\|_{L^2}^2$. Maybe one can lower down the order of derivative of $u_t$ in $\mathcal{YMH}_k(\nabla_t,u_t)$ to obtain a $C^0$-bound for $u_t$.
\end{remark}

\begin{remark}
If $\dim(M)<2(k+1)$, the long time existence of the higher order Seiberg-Witten flow studied in \cite{Sa} can be proved like Theorem \ref{mth2}.
\end{remark}

\vskip 1cm

{\bf Acknowledgement}: The author would like to thank Dr. Saratchandran for his patience on explaining the paper \cite{Sa}. The author is supported by the Fundamental Research Funds for the Central Universities (No. 19lgpy239).

\end{document}